\RequirePackage{ifpdf} 
\ifpdf
\documentclass[10pt,pdftex]{amsart}
\else
\documentclass[10pt,dvips]{amsart}
\fi

\ifpdf
\fi
\usepackage[bookmarks, 
colorlinks=false, backref,plainpages=false]{hyperref}
\usepackage[english]{babel}

\usepackage{hyperref}

\usepackage{graphics}
\usepackage{subfigure}
\usepackage{graphicx}
\usepackage{epsfig}
\usepackage{amsmath,amssymb,amsthm}
\usepackage{epstopdf}

\renewcommand{\theequation}{\thesection.\arabic{equation}}

\newtheorem{thm}{Theorem}[section]

\newtheorem{lem}[thm]{Lemma}

\newtheorem{rem}[thm]{Remark}

\usepackage[foot]{amsaddr}

\begin{document}
\newcommand{\BX}{{\bf X}}
\newcommand{\cv}{{\cal V}}
\newcommand{\cW}{{\cal W}}
\newcommand{\co}{{\cal O}}

\renewcommand{\theequation}{\thesection.\arabic{equation}}
\def\@eqnnum{{\reset@font\rm (\theequation)}}

\def\abstract{
\advance \rightskip by 10mm
\advance \leftskip by 10mm
\vspace{-0.8em}
\noindent
\small{\bf Abstract.}
}
\def\endabstract{\par\normalsize\rm}

\def\Xint#1{\mathchoice
{\XXint\displaystyle\textstyle{#1}}%
{\XXint\textstyle\scriptstyle{#1}}%
{\XXint\scriptstyle\scriptscriptstyle{#1}}%
{\XXint\scriptscriptstyle\scriptscriptstyle{#1}}%
\!\int}
\def\XXint#1#2#3{{\setbox0=\hbox{$#1{#2#3}{\int}$}
\vcenter{\hbox{$#2#3$}}\kern-.5\wd0}}
\def\ddashint{\Xint=}
\def\dashint{\Xint-}

\def\a{\alpha}
\def\b{\beta}
\def\d{\delta}\def\D{\Delta}
\def\e{\epsilon}
\def\g{\gamma}\def\G{\Gamma}
\def\k{\kappa}
\def\lam{\lambda}\def\Lam{\Lambda}
\renewcommand\o{\omega}\renewcommand\O{\Omega}
\def\s{\sigma}\def\S{\Sigma}
\renewcommand\t{\theta}\def\vt{\vartheta}
\newcommand{\vphi}{\varphi}
\def\z{\zeta}

\newcommand{\tsigma}{\tilde{\s}}
\newcommand{\tbsigma}{\tilde{\bsigma}}
\def\te{\tilde{\e}}
\def\tu{\tilde{u}}

\newcommand{\bchi}{\mbox{\boldmath$\chi$}}
\newcommand{\bdelta}{\mbox{\boldmath$\delta$}}
\newcommand{\bepsilon}{\mbox{\boldmath$\epsilon$}}
\newcommand{\bfeta}{\mbox{\boldmath$\eta$}}
\newcommand{\bgamma}{\mbox{\boldmath$\gamma$}}
\newcommand{\bomega}{\mbox{\boldmath$\omega$}}
\newcommand{\bvphi}{\mbox{\boldmath$\varphi$}}
\newcommand{\bphi}{\mbox{\boldmath$\phi$}}
\newcommand{\bPhi}{\mbox{\boldmath$\Phi$}}
\newcommand{\bpsi}{\mbox{\boldmath$\psi$}}
\newcommand{\bPsi}{\mbox{\boldmath$\Psi$}}
\newcommand{\bsigma}{\mbox{\boldmath$\sigma$}}
\newcommand{\btau}{\mbox{\boldmath$\tau$}}
\newcommand{\bxi}{\mbox{\boldmath$\xi$}}
\newcommand{\brho}{\mbox{\boldmath$\rho$}}
\newcommand{\bbeta}{\mbox{\boldmath$\beta$}}
\newcommand{\bzeta}{\mbox{\boldmath$\zeta$}}

\def\bk{\boldsymbol{\kappa}}
\def\bmu{\boldsymbol\mu}
\def\bxi{\boldsymbol{\xi}}
\def\bz{\boldsymbol{\zeta}}

\def\ba{{\bf a}}
\def\bb{{\bf b}}
\def\bc{{\bf c}}
\def\be{{\bf e}}
\def\bff{{\bf f}}
\def\bg{{\bf g}}
\def\bn{{\bf n}}
\def\bp{{\bf p}}
\def\bq{{\bf q}}
\def\bs{{\bf s}}
\def\bt{{\bf t}}
\def\bu{{\bf u}}
\def\bv{{\bf v}}
\def\bw{{\bf w}}
\def\bx{{\bf x}}
\def\by{{\bf y}}
\def\bzz{{\bf z}}

\def\bD{{\bf D}}
\def\bE{{\bf E}}
\def\bF{{\bf F}}
\def\bH{{\bf H}}
\def\bJ{{\bf J}}
\def\bV{{\bf V}}
\def\bU{{\bf U}}
\def\bW{{\bf W}}
\def\bX{{\bf X}}
\def\bY{{\bf Y}}

\def\cA{{\cal A}}
\def\cC{{\cal C}}
\def\cD{{\cal D}}
\def\cE{{\cal E}}
\def\cF{{\cal F}}
\def\cG{{\cal G}}
\def\cI{{\cal I}}
\def\cJ{{\cal J}}
\def\cK{{\cal K}}
\def\cL{{\cal L}}
\def\cO{{\cal O}}
\def\cP{{\cal P}}
\def\cQ{{\cal Q}}
\def\cR{{\cal R}}
\def\cS{{\cal \Sigma}}
\def\cT{{\cal T}}
\def\cU{{\cal U}}
\def\cV{{\cal V}}

\def\scT{{_\cT}}
\def\sD{{_D}}
\def\sE{{_E}}
\def\sF{{_F}}
\def\sFz{{_{F_z}}}
\def\sK{{_K}}
\def\sI{{_I}}
\def\sb{{_b}}
\def\sN{{_N}}

\def\curl{{{\bf curl} \ }}
\def\rot{{\mbox{rot}\ }}
\def\BPI{{\bf \Pi}}

\def\cth{\cT_h}
\def\ctH{\cT_H}

\def\tJ{\tilde{\J}}

\def\hK{\widehat{K}}
\def\hx{\widehat{x}}
\def\hy{\widehat{y}}
\def\bhv{\widehat{\bv}}

\def\l{\ell}
\def\bl{\boldsymbol{\ell}}
\def\col{\colon}
\def\f12{\frac12}
\def\dfrac{\displaystyle\frac}
\def\dint{\displaystyle\int}
\def\nab{\nabla}
\def\p{\partial}
\def\sm{\setminus}
\def\dsum{\displaystyle\sum}
\newcommand{\pp}[2]{\frac{\partial {#1}}{\partial {#2}}}
\def\bzero{{\bf 0}}

\def\divv{\nab\cdot}
\def\divx{\nab_x\cdot}
\def\divtx{\nab_{t,x}\cdot}
\def\nabx{\nab_x}

\newcommand{\grad}{\nabla}
\newcommand{\curlt}{{\nabla \times}}
\newcommand{\gperp}{\nabla^{\perp}}
\newcommand{\gradt}{\nabla\cdot}

\def\forallqq{\quad\forall\,}
\def\aph{A^{1/2}}
\def\amh{A^{-1/2}}

\def\osc{{\rm osc \, }}

\def\Im{{\rm Im}}
\newcommand{\tr}{{\rm tr}}
\def\divvr{{\rm div}}
\def\curllr{{\rm curl}}
\def\curll{{\rm curl}}
\def\curl{{\bf curl}}
\newcommand{\bgrad}{{\bf grad}}
\newcommand\diam{\mathrm{diam\,}}
\renewcommand\Im{\mathrm{Im\,}}
\def\Span{\mbox{Span}}
\def\supp{\mbox{supp\,}}
\newcommand{\trace}{{\rm trace}}

\newcommand{\tri}{|\!|\!|}
\newcommand{\ljump}{\lbrack\!\lbrack}
\newcommand{\rjump}{\rbrack\!\rbrack}
\newcommand{\bdm}{\begin{displaymath}}
\newcommand{\edm}{\end{displaymath}}
\newcommand{\beq}{\begin{equation}}
\newcommand{\eeq}{\end{equation}}
\newcommand{\beqa}{\begin{eqnarray}}
\newcommand{\eeqa}{\end{eqnarray}}
\newcommand{\beqas}{\begin{eqnarray*}}
\newcommand{\eeqas}{\end{eqnarray*}}
\newcommand{\ul}{\underline}
\newcommand{\wh}{\widehat}
\newcommand{\la}{\langle}
\newcommand{\ra}{\rangle}

\newcommand{\Lt}{L^2(\Omega)}
\newcommand{\Lts}{L^2(\Omega)^2}
\newcommand{\Ltc}{L^2(\Omega)^3}
\newcommand{\Ho}{H^1(\Omega)}
\newcommand{\Hoh}{H^1(\wh{\Omega})}
\newcommand{\Hoi}{H^1(\Omega_i)}
\newcommand{\Hos}{H^1(\Omega)^2}
\newcommand{\Hoc}{H^1(\Omega)^3}
\newcommand{\Hoch}{H^1(\wh{\Omega})^3}
\newcommand{\Hoci}{H^1(\Omega_i)^3}
\newcommand{\Hoz}{H^1_0(\Omega)}
\newcommand{\Ht}{H^2(\Omega)}
\newcommand{\Hti}{H^2(\Omega_i)}
\newcommand{\Hts}{H^2(\Omega)^2}
\newcommand{\Htc}{H^2(\Omega)^3}
\newcommand{\Htz}{H^0(\Omega)}
\newcommand{\Hh}{H^{1/2}(\Gamma)}
\newcommand{\Hhi}{H^{1/2}(\Gamma_i)}
\newcommand{\Hmh}{H^{-1/2}(\Gamma)}
\newcommand{\Hdiv}{H(\divvr;\,\Omega)}
\newcommand{\Hdivh}{H(\divv;\,\wh \Omega)}
\newcommand{\hcurl}{H(\curl\,A;\,\Omega)}
\newcommand{\Hcurl}{H(\curll\,A;\,\Omega)}
\newcommand{\Hcrl}{H(\curll\,;\,\Omega)}
\newcommand{\hcrl}{H(\curl\,;\,\Omega)}
\newcommand{\Hcrlh}{H(\curll\,;\,\wh\Omega)}
\newcommand{\hcrlh}{H(\curl\,;\,\wh\Omega)}
\newcommand{\Wdiv}{\BW_0(\mbox{\divv}\,;\,\Omega)}
\newcommand{\Wcurl}{\BW_0(\mbox{\curl}\,A;\,\Omega)}
\newcommand{\WcrossV}{\BW \times V}

\def\grad{{\nabla}}

\def\calS{{\cal S}}
\def\calT{{\cal T}}
\def\cA{{\mathcal A}}
\def\cB{{\cal B}}
\def\cD{{\mathcal{D}}}

\def\cH{{\cal H}}
\def\ba{{\mathbf{a}}}
\def\bbz{{\mathbf{z}}}

\def\beps{{\mathbf{\epsilon}}}

\def\brho{\bf{\rho}}
\def\cM{{\mathcal{M}}}
\def\cN{{\mathcal{N}}}
\def\cT{{\mathcal{T}}}
\def\cE{{\mathcal{E}}}
\def\cP{{\mathcal{P}}}
\def\cF{{\mathcal{F}}}

\def\cB{{\mathcal{B}}}
\def\cG{{\mathcal{G}}}

\def\cL{{\mathcal{L}}}
\def\cJ{{\mathcal{J}}}
\def\cK{{\mathcal{K}}}
\def\cV{{\mathcal{V}}}
\def\cW{{\mathcal{W}}}
\def\bP{{\mathbf{P}}}
\def\bS{{\mathbf{S}}}
\def\bQ{{\mathbf{Q}}}

\def\bRT{{\mathbf{RT}}}
\def\bBDM{{\mathbf{BDM}}}

\def\bSigma{{\mathbf{\Sigma}}}
\newcommand{\lJump}{[\![}
\newcommand{\rJump}{]\!]}
\newcommand{\jump}[1]{[\![ #1]\!]}

\newcommand{\sd}{\bsigma^{\Delta}}
\newcommand{\rd}{\brho^{\Delta}}

\newcommand{\eps}{\epsilon}

\def\ma{{\mathtt a}}
\def\mb{{\mathtt b}}
\def\mc{{\mathtt c}}
\def\md{{\mathtt d}}

\def\eg{{\mathtt {eg}}}
\def\full{{\mathtt {full}}}

\title [Robust Augmented Mixed FEMs for Stokes Interface Problems]{Robust Augmented Mixed FEMs for Stokes Interface Problems with Discontinuous Viscosity in Multiple Subdomains }
\author[Y. Liang and S. Zhang]{Yuxiang Liang and Shun Zhang}
\address{Department of Mathematics, City University of Hong Kong, Kowloon Tong, Hong Kong, China}
\email{yuxiliang7-c@my.cityu.edu.hk, shun.zhang@cityu.edu.hk}
\thanks{This work was supported in part by
Research Grants Council of the Hong Kong SAR, China, under the GRF Grant Project No. CityU 11316222}
\date{\today}


\maketitle
\begin{abstract}
A stationary Stokes problem with a piecewise constant viscosity coefficient in multiple subdomains is considered in the paper. For standard finite element pairs, a robust inf-sup condition is required to show the robustness of the discretization error with respect to the discontinuous viscosity, which has only been proven for the two-subdomain case in the paper [Numer. Math. (2006) 103: 129–149]. To avoid the robust inf-sup condition of a discrete finite element pair for multiple subdomains, we propose an ultra-weak augmented mixed finite element formulation.   By adopting a Galerkin-least-squares method, the augmented mixed formulation can achieve stability without relying on the inf-sup condition in both continuous and discrete settings. The key step to have the robust a priori error estimate is that two norms, one energy norm and one full norm, are used in the robust continuity. The robust coercivity is proved for the energy norm. A robust a priori error estimate in the energy norm is then derived  with the best approximation property in the full norm for the case of multiple subdomains.
Additionally, the paper introduces a singular Kellogg-type example with exact solutions for the first time. Extensive numerical tests are conducted to validate the robust error estimate. 
\end{abstract}

\section{Introduction}\setcounter{equation}{0}
In this paper, we consider the Stokes interface boundary value problem in multi-subdomains.
We assume that, $\{\O_i\}_{i=1}^L$ is a partition of a bounded polygonal domain $\O$ with $\O_i$ being an open bounded polygonal subdomain. Let $\nu$ be the viscosity coefficient such that $\nu$ equals a constant $2\mu_i$ on the subdomain $\O_i$, where $\mu_i$ is the dynamic/shear viscosity, and we assume that 
$$
0< \lambda < \mu_i < \Lambda \quad \forall i = 1\cdots L.
$$
The velocity-pressure formulation of the Stokes equation is 
\beq
\label{vp0}
\left\{
\begin{array}{lllll}
- \gradt(\nu \beps(\bu)) +\nabla p &=& \bff   & \mbox{in } \O,
 \\[1mm]
 \gradt \bu &=& 0 & \mbox{in } \O,
 \\[1mm]
\bu &=& 0 & \mbox{on } \partial\O,
\end{array}
\right.
\eeq
where $\beps(\bu)$ is the strain tensor.
As discussed in \cite{OR:06} and \cite{SG:17}, the motivation of studying this type of problem is the multi-phase incompressible flow, where Stokes equations with discontinuous density and viscosity coefficients appear, see \cite{GR:11}. For the two-subdomain case, Olshanskii and Reusken \cite{OR:06} proved robust inf–sup results for both the continuous and discrete cases for the velocity-pressure formulation in the sense that the inf-sup constants are uniform with respect to the jump in the viscosity coefficient. However, extending these robust results to the multi-subdomain case remains an open challenge. Song and Gao \cite{SG:17} only proved the inf-sup constant depending on the second smallest viscosity constant, which is not a robust result. 

To handle the multi-subdomain case, we seek ways to avoid the inf-sup condition. One such approach using finite element techniques without inf-sup conditions is the least-squares finite element method (LSFEM)  \cite{CLMM:94,Jiang:98,BG:09,CFZ:15,LZ:18,LZ:19,QZ:20,Zhang:23,LZ:23}. However, as discussed in \cite{LZ:24}, the LSFEM struggles with robustness even in simpler cases like diffusion problems with discontinuous coefficients. In addition to the conventional LSFEM, another way that uses the least-squares philosophy is the Galerkin-Least-Squares (GaLS) method. A specialized GaLS method known as the augmented mixed finite element method, first introduced in \cite{MH:02}, offers automatic stability without the imposition of inf-sup conditions on discrete approximation spaces. Earlier contributions of GaLS methods can be found in \cite{Franca:87,FH:88}. The group of Gatica made many important and seminal contributions on developing the augmented mixed finite element method to many problems, see for example \cite{Gatica:06,BGGH:06,FGM:08,CGOT:16,CGO:17,AGR:20}. In \cite{LZ:24}, we show that for the diffusion problem with discontinuous coefficients, two versions of augmented mixed finite element methods have robust a priori and a posteriori error estimates. In this paper, we generalize the result in \cite{LZ:24} to the case of the Stokes interface problem with multi-subdomains.

In the augmented mixed finite element method, we start from a mixed stress-velocity formulation and add least-squares terms to the formulation to enhance stability. For the Stokes interface problem with discontinuous coefficients, the symmetric stress tensor $\nu\beps(\bu) = \frac{1}{2}\nu(\nabla \bu + ( \nabla \bu )^t)$ plays a central role. Due to the fact that $\nu$ is a piecewise constant function, $\nu$ cannot be taken out of the differential operator, the popular pseudostress formulation \cite{FGM:08,CTVW:10,CWZ:10,CZ:12stokes} cannot be used. We need to handle the symmetric $\nu\beps(\bu)$ directly. On the one hand, a common finite element to approximate the stress tensor is given by a un-symmetric row-wise approximation using Raviart–Thomas elements. The symmetry is often imposed weakly by letting it orthogonal to all skew-symmetric tensors, which is a new variable, the vorticity tensor (the skew-symmetric part of the velocity gradient) as an additional unknown. This may lead to new inf-sup requirements. In \cite{AGR:20}, a new formulation is proposed for the linear elasticity where the symmetry is imposed ultra-weakly. We modify the approach in  \cite{AGR:20} to the case of the Stokes interface problem with multi-subdomains. In our augmented formulation,  consistent least-squares terms are added to the ultra-weak mixed formulation. A key step of our robust a priori analysis is in the continuity of the bilinear form, we use two different norms: a viscosity-dependent energy norm $\tri \cdot \tri_{\eg,\theta}$ (defined in \eqref{norm_eg}) and a viscosity-dependent full norm $\tri \cdot \tri_{\full,\theta}$ (defined in \eqref{norm_full}), see Lemma \ref{lem_con}. The robust coercivity is proved with respect to the viscosity-dependent energy norm. Then we have the robust a priori error estimate \eqref{Cea}. It is important to realize that on the left-hand side of \eqref{Cea}, the norm is the energy norm, while on the right-hand side of \eqref{Cea}, the norm is a standard $\nu$-weighted norm, where we have standard approximation properties.  

For the a posteriori error estimator, as discussed in \cite{LZ:24}, the standard least-squares functional error estimator \cite{CLW:04} can be used as a posteriori error estimator. On the other hand, it is hard to prove the robustness with respect to $\nu$ for the Stokes problem in the current setting. Nevertheless, numerical experiments are given to show that they are indeed $\nu$-robust for the test numerical problems.

An additional important contribution of this paper is the construction of a Kellogg-type example with an exact solution for the first time. Kellogg \cite{Kellogg:74} constructed examples with exact solutions for the diffusion problem with discontinuous coefficients.  This example has been widely adopted as a benchmark test to evaluate various a posteriori error estimators in the literature, such as \cite{MNS:02,CD:02,Pet:02,CZ:09,Mit:13,adaptiveVEM:23}.
However, for the Stokes problem, a comparable test problem has been absent until now.  We present a detailed derivation of a nonlinear system to generate a series of examples with exact solution and singularity due to intersecting interfaces for the first time. The solution is divergence free and can have different regularity by tuning the coefficients. This  example not only fills a gap in the field but also serves as a valuable tool for evaluating the performance of various numerical methods for solving the Stokes equation.

In this paper, for the simplicity of theoretical analysis, we assume that the interface is aligned with the computational mesh. For the case more complicated Stokes interface problems in the curved interface settings, cut finite element methods \cite{HLZ:14}, Xfem extended finite element methods \cite{KGR:16}, and a first-order system least-squares method \cite{Ber:18} are developed. 

For the simpler diffusion problem with discontinuous coefficients, the robust (but not local optimal) a priori  and robust residual type a posteriori error estimate for the energy norm was obtained for the conforming FEM \cite{BeVe:00}. Robust and local optimal {a priori} error estimates have also been derived for mixed FEM in \cite{Zhang:20mixed} and nonconforming FEM and discontinuous Galerkin FEM in \cite{CHZ:17} without a restrictive assumption on the distribution of the coefficients. We also present a detailed discussion of the robust and local optimal {a priori} error estimate for the conforming FEM in \cite{CHZ:17}. For recovery-based error estimators, robust a posteriori error estimates are obtained by us in \cite{CZ:09,CZ:10a,CYZ:11}. Robust equilibrated error estimators were developed by us in \cite{CZ:12,CCZ:20,CCZ:20mixed}. Robust residual-type of error estimates without a restrictive assumption on the distribution of the coefficients is developed for nonconforming and DG approximations in \cite{CHZ:17}.

The rest of the paper is organized as follows. In Section 2, we introduce some notations and discuss the finite element spaces, especially the approximation property of the $\nu$-weighted space. We discuss the Stokes interface problem with some equivalent formulations in Section 3. In Section 4, the energy norm that plays a central role in the robust analysis is discussed. The augmented mixed formulations are presented in Section 5 and its finite element approximations are discussed in Section 6. In Section 7, we discuss the related least-squares formulations and a posteriori error analysis. Numerical test are conducted in Section 8. Some final comments are given in Section 9. In the appendix, we give a detailed derivation of the nonlinear system to generate the Kellogg-type example with exact solution.

\section{Notations and Finite Element Spaces}\setcounter{equation}{0}
\subsection{Notations}
For $d=2$ or $3$, let $\bv =(v_1,\, ... \, , \, v_d)^t$ be a vector function, $\btau=\left(\tau_{ij}\right)_{d\times d}$ be a tensor function, and denote $\btau$'s $i^{th}$-row by $\btau_i=(\tau_{i1},...,\tau_{id})$ for $i=1,...,d$. We define
\[
 \nabla \bv=
 \left(\dfrac{\p v_i}{\p x_j}\right)_{d\times
 d}\quad\mbox{and}\quad 
 \gradt\btau=(\gradt\btau_1,\,...\, ,\,\gradt\btau_d)^t.
 \]
The dot product of a vector $\bv$ and a tensor $\btau$ is defined by
$
\bv \cdot \btau = (\bv\cdot \btau_1, ..., \bv\cdot \btau_d)^t.
$
So we can define the normal of $\btau$, $\bn \cdot \btau = (\bn\cdot \btau_1,...,\bn \cdot \btau_d)^t$. We also define the identity tensor $ I = (\delta_{ij})_{d\times d}$, where the Kronecker $ \delta_{ii} = 1 $ and $\delta_{ij} =0$, if $i\neq j$. The inner product between tensors is defined by
$
\bsigma : \btau = \sum_{i=1}^d\sum_{j=1}^{d}\sigma_{ij}\tau_{ij}.
$
The trace of a tensor is defined by $ \tr \btau = \btau : I = \sum_{i=1}^d \btau_{ii}$.

Define $H(\divvr;\O) = \{\btau\in L^2(\O)^{d},~\gradt \btau\in L^2(\O)\}$ with the norm $\|\btau\|_{H(\divvr,\O)} = (\|\btau\|^2_0+\|\gradt\btau\|^2_0)^{1/2}$. We use the notation $\bH(\divvr;\O):=H(\divvr;\O)^d$ to denote the space whose members are tensors with row vectors in $H(\divvr;\O)$. For a function $\nu>0$, its weighted subspace is 
\beq
\bH_\nu(\divvr;\O) = \{\btau\in H(\divvr;\O)^d;\int_\O \nu^{-1}\tr(\btau) dx =0\}.
\eeq
Define $
\bH^1_0(\O):= H_0^1(\O)^d = \{\bv \in H^1(\O)^d; \bv = 0 \mbox{ on } \partial\O\}$ 
and
$L^2_{\nu}(\O) = \{q\in L^2(\O);~(\nu^{-1}q,1)_\O = 0\}$.

\subsection{Finite element spaces}
Let $\cT = \{K\}$ be a triangulation of $\O$ using simplicial elements. The mesh $\cT$ is assumed to be regular.  Furthermore, assume that interfaces do not cut through any element $K\in\cT$. Let $P_k(K)$ for $K\in\cT$ be the space of polynomials of degree $k$ on an element $K$. Denote the standard linear and quadratic conforming finite element spaces by 
$$
S_{1,0} = \{v\in  H_0^1(\O): v|_K \in P_1(K),~\forall~K\in\cT \}
\quad\mbox{and}\quad
S_{2,0} = \{v\in H_0^1(\O): v|_T \in P_2(K),~\forall~K\in\cT\},
$$
respectively. 

Denote the local lowest-order Raviart-Thomas (RT) \cite{RT:77} on element $K\in\cT$ by $RT_0(K)=P_0(K)^d +\bx\,P_0(K)$. Then the lowest-order $\Hdiv$ conforming RT space is defined by
 \[
 RT_{0}=\{\btau\in H(\divvr;\Omega):
 \btau|_K\in RT_0(K)\,\,\,\,\forall\,\,K\in\cT\}.
 \] 
Similarly,  the lowest-order  $\Hdiv$-conforming Brezzi-Douglas-Marini (BDM) space is defined by
$$
BDM_{1} = \{\btau\in H(\divvr;\O): \btau|_K \in P_1(K)^d,~\forall~K \in\cT\}.
$$
We discuss some \noindent{local approximation results}. By Sobolev's embedding theorem, $H^{1+s}(\O)$, with $s>0$ for two dimensions and $s > 1/2$ for three dimensions, is embedded in $C^0(\O)$. Thus, we can define the nodal interpolation $I^{nodal}_h$ of a function $v\in H^{1+s}(\O)$ with $I^{nodal}_hv \in S_{1,0}$ and $I^{nodal}_hv(z) = v(z)$ for a vertex $z\in \cN$. It is important to notice that the nodal interpolation is completely element-wisely defined.  
We have the following local interpolation estimate for the linear nodal interpolation $I^{nodal}_h$ with local regularity $0<s_K\leq 1$ in two dimension and $1/2<s_K\leq 1$ in there dimension \cite{DuSc:80,CHZ:17}:
\beq \label{nodal_inter}
\|\nabla(v- I^{nodal}_h v)\|_{0,K} \leq C h_K^{s_K}|\nabla v|_{s_K,K} .
\quad\mbox{and}\quad
\|v- I^{nodal}_h v\|_{0,K} \leq C h_K^{1+s_K}|\nabla v|_{s_K,K}.
\eeq

Assume that $\btau\in L^r(\O)^d\cap H(\divvr;\O)$, and locally $\btau \in H^{s_K}(K)$ with the local regularity  $1/2<s_K\leq 1$. Let $I^{rt}_h$ be the canonical RT interpolation from $L^r(\O)^d\cap H(\divvr;\O)$ to $RT_{0}$. Then the following local interpolation estimates hold for local regularity $1/2<s_K\leq 1$ with the constant $C_{rt}$ being unbounded as $s_K\downarrow 1/2$ (see Chapter 16 of \cite{FE1}): 
\beq \label{RT_inter1}
\|\btau- I^{rt}_h\btau\|_{0,K} \leq C_{rt} h_K^{s_K}|\btau|_{s_K,K} \quad \forall K\in \cT.
\eeq
Due to the commutative property of the standard RT interpolation, if we further assuming that $\gradt \btau|_K \in H^{t_K}(K)$, $0<t_K\leq 1$, then
\beq \label{RT_inter2}
\|\gradt (\btau- I^{rt}_h\btau)\|_{0,K} \leq C h_K^{t_K}|\gradt\btau|_{t_K,K} \quad \forall K\in \cT.
\eeq

For $S_2$ and $BDM_1$, we will only use them when the regularity of the solution is high. For $v\in H^3(\O)$, the following local interpolation result in standard for nodal interpolation $I_h$ in $S_{2,0}$,
\beq \label{nodal_inter2}
\|\nabla(v- I_h v)\|_{0,K} \leq C h_K^{2}|v|_{3,K}.
\eeq
Also, assume that $\btau \in H^2(\O)^d$, we have the following standard interpolation result for the $BDM_{1}$, assuming that $I^{bdm}_h$ is the standard $BDM_{1}$ interpolation,
\beq \label{BDM_inter1}
\|\btau- I^{bdm}_h\btau\|_{0,K} \leq C_{bdm} h_K^{2}|\btau|_{2,K} \quad \forall K\in \cT.
\eeq
If we further assuming that $\gradt \btau|_K \in H^{t_K}(K)$, $0<t_K\leq 1$, then
\beq \label{BDM_inter2}
\|\gradt (\btau- I^{bdm}_h\btau)\|_{0,K} \leq C h_K^{t_K}|\gradt\btau|_{t_K,K} \quad \forall K\in \cT.
\eeq

\subsection{Approximation properties in weighted spaces}
We then discuss interpolation results in $\bH_{\nu}(\divvr;\O)$.
Define 
\beq
\bRT_{0,\nu} = (RT_0)^d\cap \bH_{\nu}(\divvr;\O) \quad\mbox{and}\quad
\bBDM_{1,\nu} = (BDM_1)^d\cap \bH_{\nu}(\divvr;\O)
\eeq
Define the interpolation operator $I^{rt}_{h,\nu}: \bH_{\nu}(\divvr;\O)\rightarrow \bRT_{0,\nu}$
and $I^{bdm}_{h,\nu}: \bH_{\nu}(\divvr;\O)\rightarrow \bBDM_{0,\nu}$ as follow:
\beq\label{Irtbdm}
I^{rt}_{h,\nu} \btau := I^{rt}_{h}\btau - \phi(I^{rt}_{h}\btau,\nu)  I \mbox{ and }
I^{bdm}_{h,\nu} \btau := I^{bdm}_{h}\btau - \phi(I^{bdm}_{h}\btau,\nu)  I
 \mbox{ with  }\phi(\btau,\nu) = \frac1{d(1,\nu^{-1})} (\tr \btau,\nu^{-1}).
\eeq
It is easy to check that $I^{rt}_{h,\nu} \btau  \in \bRT_{0,\nu}$ and $I^{bdm}_{h,\nu} \btau  \in \bBDM_{0,\nu}$ due to
$$
(\tr (I^{rt}_{h,\nu}\btau),\nu^{-1}) = (\tr (I^{rt}_{h,\nu}\btau-\phi(I^{rt}_{h}\btau,\nu) I ),\nu^{-1}) 
= (\tr (I^{rt}_{h}\btau), \nu^{-1}) -  \dfrac{d(1,\nu^{-1})}{d(1,\nu^{-1})} (\tr (I^{rt}_{h}\btau),\nu^{-1})= 0.
$$

For $\btau  \in \bH_{\nu}(\divvr;\O)$,  by \eqref{RT_inter2}, if we assume that $\gradt \btau|_K \in H^{t_K}(K)$, $0<t_K\leq 1$, then
\beq\label{ine_interPi_div}
\|\nu^{-1/2}\gradt(\btau-I^{rt}_{h,\nu}  \btau)\|_{0,K} = \|\nu^{-1/2}\gradt(\btau-I^{rt}_{h} \btau)\|_{0,K}
\leq C \nu_K^{-1/2}h^{t_K}_K|\gradt\btau|_{t_K,K},  \quad \forall K\in \cT.
\eeq
For $\btau  \in \bH_{\nu}(\divvr;\O)$, we have $(\tr \btau,\nu^{-1})=0$, thus  $\phi(\btau,\nu)=0$. 
For the term $\|\nu^{-1/2}(\btau- I^{rt}_{h,\nu}  \btau)\|_0$, we then have
\begin{eqnarray*}
\|\nu^{-1/2}(\btau- I^{rt}_{h,\nu}  \btau)\|_0
&=& \|\nu^{-1/2}(\btau - I^{rt}_{h} \btau + \phi(I^{rt}_{h}\btau,\nu)  I)\|_0 = \|\nu^{-1/2}(\btau - I^{rt}_{h} \btau + \phi(I^{rt}_{h}\btau-\btau,\nu)  I)\|_0\\
&\leq& \|\nu^{-1/2}(\btau-\Pi_{rt}\btau)\|_0
+ 
\frac1{(1,\nu^{-1})} \|\nu^{-1/2}(\tr ( I^{rt}_{h} \btau - \btau),\nu^{-1}) \|_0.
\end{eqnarray*}
We estimate of the second term first. We have
\begin{eqnarray*}
\|\nu^{-1/2}(\tr ( I^{rt}_{h} \btau - \btau),\nu^{-1}) \|_0 &\leq &  \|\nu^{-1/2}\|_0 |(\tr ( I^{rt}_{h} \btau - \btau),\nu^{-1})|\\
&=&   \|\nu^{-1/2}\|_0 |(\nu^{-1/2} \tr(I^{rt}_{h} \btau - \btau),\nu^{-1/2})| \leq \|\nu^{-1/2}\|_0^2  \|\nu^{-1/2} (\btau -I^{rt}_{h}\btau)\|_0.
\end{eqnarray*}
Then it is true that for $\btau  \in \bH_{\nu}(\divvr;\O)$,
\beq\label{interbdmPi_0}
\|\nu^{-1/2}(\btau-I^{rt}_{h,\nu}\btau)\|_0 \leq 2\|\nu^{-1/2}(\btau-I^{rt}_{h}\btau)\|_0
\leq 2C_{rt} \sum_{K\in\cT}\nu_K^{-1/2}h^{s_K}_K|\btau|_{s_K,K}.
\eeq

Similarly, assume that $\btau \in H^2(\O)^d\cap  \bH_{\nu}(\divvr;\O)$, we have
\beq \label{BDMPi_inter1}
\|\nu^{-1/2}(\btau-I^{bdm}_{h,\nu}\btau)\|_0 \leq 2C_{bdm} \sum_{K\in\cT}\nu_K^{-1/2}h^{2}_K|\btau|_{2,K}.
\eeq
If we further assuming that $\gradt \btau|_K \in H^{t_K}(K)$, $0<t_K\leq 1$, then
\beq \label{BDMPi_inter2}
\|\nu^{-1/2} \gradt (\btau- I^{bdm}_{h,\nu}\btau)\|_{0,K} \leq C \nu_K^{-1/2} h_K^{t_K}|\gradt\btau|_{t_K,K} \quad \forall K\in \cT.
\eeq

\section{Stokes Interface Problems: Some Equivalent Formulations} \setcounter{equation}{0}
In this section, we discuss some physically meaningful equivalent formulations for the Stokes interface problems.
%
For a $\bv \in H^1(\O)^d$, let $\nabla \bv$ be its gradient tensor. Define the symmetric part of $\nabla \bv$ by $\beps(\bv) = \frac{1}{2}(\nabla \bv + ( \nabla \bv )^t)$. 
Let $p$ be the pressure, and  $\beps(\bu) = \frac{1}{2}(\nabla \bu + ( \nabla \bu )^t)$ the velocity deformation tensor, also called the strain tensor, where $\nabla \bu$ is the velocity gradient tensor. Let $\bsigma$ be the stress tensor, then the constitutive law for incompressible linear Newtonian fluids is
$ \bsigma - \nu\beps(\bu) + pI  = 0$  and  $\gradt \bu=0$.
In this paper, we only consider the stationary problem, thus, the linear momentum equation is $\gradt\bsigma     = -\bff$.
For simplicity, we assume that the homogeneous boundary Dirichlet condition: $\bu = 0$ on  $\p\O$.
Thus we have the following stress-velocity-pressure formulation with discontinuous coefficient $\nu$, $\bu\in \bH^1_0(\O)$ satisfying
\begin{equation}
\label{stress_velocity_pressure}
\left\{
\begin{array}{lllll}
\gradt\bsigma    & =& -\bff & \mbox{in } \O,
 \\[1mm]
 \bsigma - \nu\beps(\bu) + pI & =& 0 & \mbox{in } \O,
 \\[1mm]
 \gradt \bu &=& 0 & \mbox{in } \O.
\end{array}
\right.
\end{equation}
Substitute the second equation of \eqref{stress_velocity_pressure} into the first equation to eliminate $\bsigma$, we have the velocity-pressure equation, $\bu\in \bH^1_0(\O)$ satisfying
\beq
\label{velocity_pressure}
\left\{
\begin{array}{lllll}
- \gradt(\nu \beps(\bu)) +\nabla p &=& \bff   & \mbox{in } \O,
 \\[1mm]
 \gradt \bu &=& 0 & \mbox{in } \O.
\end{array}
\right.
\eeq
The weak formulation of \eqref{velocity_pressure} is: find  $\bu\in \bH_0^1(\O)$ and $p\in L^2_{\nu}(\O)$ such that
\beq
\label{velocity_pressure_weak}
\left\{
\begin{array}{lllll}
(\nu\beps(\bu),\beps(\bv)) - (\gradt \bv, p) &=& (\bff,\bv) \quad &\forall \bv \in \bH_0^1(\O),
\\[1mm]
(\gradt \bu,q)&=&0 \quad &\forall q \in L^2_{\nu}(\O).
\end{array}
\right.
\eeq
For the strain $\beps(\bu)$, with the incompressibility assumption, its trace is zero, $\gradt \bu =0$. Thus $\tr\bsigma = -dp$. We then have $(1/\nu, \tr\bsigma) =0$ and $\bsigma \in \bH_\nu(\divvr;\O)$.
Eliminating $p$ from the second equation of \eqref{stress_velocity_pressure}, we get
$ \cA\bsigma = \nu\beps(\bu) $
where $\cA$ is a linear but not invertible map defined by 
$
\cA\bxi =\bxi - \frac1d\tr(\bxi)I$ for $\bxi\in L^2(\O)^{d\times d}$.
By the definition of $\cA,$ it is true that 
$
(\cA\bxi,\btau) = (\cA\bxi, \cA\btau) = (\bxi,\cA\btau)$,  
for $\bxi,\btau\in L^2(\O)^{d\times d}$.
Thus we have 
\beq\label{ine_cA}
 \|\cA \btau\|_0 = \sqrt{(\cA \btau,\btau)} = \sqrt{(\cA \btau,\cA\btau)} \quad\mbox{and}\quad
(\cA\bxi,\btau)\leq \|\cA\bxi\|_0 \|\cA\btau\|_0, \quad \bxi,\btau\in L^2(\O)^{d\times d}.
\eeq
We have the stress-velocity formulation with the discontinuous viscosity, $\bu\in \bH^1_0(\O)$ and $\bsigma \in \bH_\nu(\divvr;\O)$ satisfying 
\begin{equation}\label{stress-velocity}
\left\{
\begin{array}{lllll}
\gradt \bsigma    & =& -\bff & \mbox{in } \O, & \mbox{the equilibrium equation},
 \\[1mm]
\cA\bsigma  & =& \nu\beps(\bu) & \mbox{in } \O, & \mbox{the constitutive equation}.
\end{array}
\right.
\end{equation}
We will mainly use this formulation to derive our ultra-weak augmented mixed finite element methods.

\begin{rem}
It is noted that when $\nu$ is a global constant, several different but equivalent forms can be derived for the velocity-pressure formulation \eqref{velocity_pressure}:
\begin{eqnarray}
 \gradt(\nu \beps(\bu)) = \nu \Delta \bu = - \nu \curl (\curl \bu) = \nu (\nabla (\gradt \bu)-\curl (\curl \bu)).
\end{eqnarray}
But as stated in p.58 of Gunzburger \cite{Gun:89} or p.45 of Landau and Lifshitz \cite{LL:fluid}, when $\nu$ is not a constant, the viscosity $\nu$ can not be taken out of the differential operator and only $\gradt(\nu \beps(\bu))$ is the meaningful choice since it is derived from the principle of conservation of the linear momentum $\gradt\bsigma     = -\bff$  and the Newton-Poisson constitutive equation $\bsigma - \nu\beps(\bu) + pI=0$.
\end{rem}

\begin{rem} {\bf Surface tension force on interface.}
Let $\Gamma_{ij}$ be the interface of $\O_i$ and $\O_j$, with $i<j$, and let $\bn_{ij}$ be the unit normal to $\Gamma_{ij}$, outward-directed with respect to $\O_i$. We define $\jump{v}|_{\Gamma_{ij}} = v|_{\O_i}-v|_{\O_j}$.
In our formulation, we take a simple model assuming that $\jump{\bsigma\cdot\bn_{ij}}|_{\Gamma_{ij}} = 0$ along the interface of different subdomains. For the multi-phase flow \cite{GR:11}, due to the fact that there are different molecules with different attractive forces on both sides of the interface, there is a surface tension force at the interface. Thus, we need a jump condition on the stress:
\beq \label{interfacesigma}
\jump{\bsigma\cdot\bn_{ij}} = -\tau\kappa \bn_{ij} \quad \mbox{  on } \Gamma_{ij},
\eeq 
where $\kappa$ is the curvature of the interface and $\tau$ is a constant surface tension coefficient. Thus, the full stress-velocity-pressure model is, $\bu\in \bH^1_0(\O)$:
\begin{equation}
\label{stress_velocity_pressure_2}
\left\{
\begin{array}{rllll}
\gradt\bsigma_i    & =& -\bff & \mbox{in } \O_i,
 \\[1mm]
 \bsigma_i - \nu\beps(\bu) + pI & =& 0 & \mbox{in } \O_i,
 \\[1mm]
 \jump{\bn_{ij}\cdot\bsigma} &=& -\tau\kappa \bn_{ij} & \mbox{on } \Gamma_{ij},
  \\[1mm]
   \jump{\bu} &=& 0& \mbox{on } \Gamma_{ij},
   \\[1mm]
 \gradt \bu &=& 0 & \mbox{in } \O.
\end{array}
\right.
\end{equation}
Since the system \eqref{stress_velocity_pressure_2} is linear, we can always construct a symmetric $\hat{\bsigma}$ (or construct it approximately), such that
\begin{equation}
\label{stress_velocity_pressure_hat}
\left\{
\begin{array}{lllll}
\gradt\hat{\bsigma}_i    & =& 0 & \mbox{in } \O_i,
 \\[1mm]
 \hat{\bsigma_i}\cdot\bn_{ij} &=& -\tau\kappa \bn_{ij} & \mbox{on } \Gamma_{ij},
  \\[1mm]
  \hat{\bsigma}_j\cdot\bn_{ij} &=& 0 & \mbox{on } \Gamma_{ij},
\end{array}
\right.
\end{equation}
Then for the modified stress, $\tilde{\bsigma}:=\bsigma-\hat{\bsigma}$, we have the stress-velocity-pressure formulation, 
\begin{equation}
\label{stress_velocity_pressure_tilde}
\left\{
\begin{array}{rllll}
\gradt\tilde{\bsigma}    & =& -\bff & \mbox{in } \O,
 \\[1mm]
 \tilde{\bsigma} - \nu\beps(\bu) + pI & =& -\hat{\bsigma} & \mbox{in } \O,
 \\[1mm]
 \gradt \bu &=& 0 & \mbox{in } \O,
\end{array}
\right.
\end{equation}
or the stress-velocity formulation,
\begin{equation}
\label{stress_velocity_tilde}
\left\{
\begin{array}{rllll}
\gradt\tilde{\bsigma}    & =& -\bff & \mbox{in } \O,
 \\[1mm]
 \cA\tilde{\bsigma} - \nu\beps(\bu)& =& -\hat{\bsigma} & \mbox{in } \O,
 \\[1mm]
 \gradt \bu &=& 0 & \mbox{in } \O.
\end{array}
\right.
\end{equation}
The velocity-pressure weak formulation of \eqref{stress_velocity_pressure_2} is: find  $\bu\in \bH_0^1(\O)$ and $p\in L^2_{\nu}(\O)$ such that
\beq
\label{velocity_pressure_weak_2}
\left\{
\begin{array}{lllll}
(\nu\beps(\bu),\beps(\bv)) - (\gradt \bv, p) &=& (\bff,\bv)  +\sum_{\mbox{all interfaces $\Gamma_{ij}$}}(\tau \kappa,\bv\cdot\bn_{ij})_{\Gamma_{ij}}\quad \forall \bv \in \bH_0^1(\O),
\\[2mm]
(\gradt \bu,q)&=&0 \quad \forall q \in L^2_{\nu}(\O).
\end{array}
\right.
\eeq
Note that the lefthand side is the same as that of \eqref{velocity_pressure_weak}.
Similarly, the jump condition will only affect the righthand side of the corresponding variational problem of different formulations, which is not going to be a problem of the robust a priori error analysis. Thus, we stick to the simple problem without surface force in this paper.
\end{rem}

\begin{rem}
In \cite{CLW:04,CTVW:10}, the pseudostress-velocity formulation is introduced for the Stokes equation. The main advantage of the pseudostress-velocity formulation is that the symmetric condition of the stress is not required. Unfortunately, it cannot be applied to the discontinuous viscosity case here. The pseudostress is defined as $\bsigma_{ps} = \nu \nabla \bu -pI = \bsigma - \nu (\nabla \bu)^t$. For a constant $\nu$, we have $\gradt ((\nabla \bu)^t) =0$, thus, $\gradt\bsigma_{ps} = \gradt\bsigma = -\bff$. On the other hand, for a piecewise constant function $\nu$,  we cannot take out $\nu$, and thus $\gradt (\nu(\nabla \bu)^t) \neq 0$. The constitutive law  $\gradt\bsigma_{ps} = -\bff$ does not hold. Another reason that the pseudostress formulation cannot be used is the interface condition of $\bsigma$ \eqref{interfacesigma} cannot be changed to the pseudostress.
\end{rem}

\section{On the right energy norm} \setcounter{equation}{0}
To drive the robust analysis, we need to choose the right $\nu$-dependent norm. 
For the Stokes problem in the velocity-pressure form \eqref{velocity_pressure_weak}, let $V_{d0}$ be the divergence free space:
$
V_{d0} :=\{\bv\in \bH_0^1(\O): \gradt\bv =0\}. 
$
In $V_{d0}$, the weak problem is simplified as: find $\bu \in V_{d0}$, such that
\beq \label{epsbudiv0}
(\nu\beps(\bu),\beps(\bv)) = (\bff,\bv) \quad \forall \bv \in V_{d0},
\eeq
The corresponding energy minimization problem is: find $\bu \in V_{d0}$, such that
$$
L(\bu) = \min_{\bv\in V_{d0}} L(\bv) \quad \mbox{with the energy functional defined as } L(\bv):= \frac12\|\nu^{1/2}\beps (\bv)\|_0^2-(\bff,\bv)\quad \forall \bv\in V_{d0}.
$$
Thus, the natural energy norm we should use for the Stokes interface problem is 
$\|\nu^{1/2}\beps(\bv)\|_0$. 
Similar to the situation of the diffusion problem, we have the following first Korn inequality (see p.318 of \cite{BrSc:08}): $\|\beps(\bv)\|_0 \geq C \|\nabla \bv\|_0$, for $\bv \in \bH_0^1(\O)$. We do not have a $\nu$-robust version of the first Korn inequality for a discontinuous $\nu$. 

To get the right norm of $\bsigma$, we notice that $\nu^{-1/2}\cA \bsigma = \nu^{1/2}\beps( \bu)$, thus the appropriate energy norm for $\btau\in \bH_\nu(\divvr;\O)$ is $\|\nu^{-1/2}\cA\btau\|_0$. But $\|\nu^{-1/2}\cA\btau\|_0$ alone is not a norm of $\bH_\nu(\divvr;\O)$. Luckily, we have the following inequality for some $C>0$, see \cite{CTVW:10,BBF:13},
\beq \label{cAineq}
C\|\btau\|_0 \leq \|\cA \btau\|_0 + \|\gradt\btau\|_0 \quad \forall \btau \in \bH_{\nu}(\divvr;\O).
\eeq
With the help of \eqref{cAineq},  we need to add some divergence part to $\|\nu^{-1/2}\cA\btau\|_0$ to get a norm for $\btau \in \bH_\nu(\divvr;\O)$. Thus, we choose $\nu$-weighted $H(\divvr)$-norm $(\|\nu^{-1/2}\cA\btau\|_0^2+ \|\nu^{-1/2}\gradt\btau\|_0^2)^{1/2}$ or a mesh and $\nu$-weighted $H(\divvr)$-norm $(\|\nu^{-1/2}\cA\btau\|_0^2+ \|\nu^{-1/2}h\gradt\btau\|_0^2)^{1/2}$. It is important to notice that even though $(\|\nu^{-1/2}\cA\btau\|_0^2+ \|\nu^{-1/2}\gradt\btau\|_0^2)^{1/2}$ or $(\|\nu^{-1/2}\cA\btau\|_0^2+ \|\nu^{-1/2}h\gradt\btau\|_0^2)^{1/2}$ are norms for $\bH_\nu(\divvr;\O)$, the equivalence of them with the standard $H(\divvr)$-norm may depend on $\nu$ or the mesh-size.

With the above discussion, we use the following energy norm in our analysis:
\beq \label{norm_eg}
\tri(\btau,\bv)\tri_{\eg,\theta} := (\|\nu^{1/2}\eps(\bv)\|^2_0 + \|\nu^{-1/2}\cA\btau\|_0^2 +  \|\sqrt{\theta/\nu}\gradt\btau\|^2_0)^{1/2},
\quad \forall (\btau,\bv) \in \bH_\nu(\divvr;\O) \times \bH_0^1(\O).
\eeq
Two cases of $\theta$ will be used, one is $\theta=1$ and the other case is that $\theta$ is a mesh-size function such that $\theta|_K = h_K^2$, for $K\in\cT$.
We will also use the full norm
\beq \label{norm_full}
\tri(\btau,\bv)\tri_{\full,\theta} := (\|\nu^{1/2}\nabla \bv\|^2_0+\|\sqrt{\nu/\theta} \bv\|^2_0 + \|\nu^{-1/2}\btau\|_0^2 +  \|\sqrt{\theta/\nu}\gradt\btau\|^2_0)^{1/2},
\quad \forall (\btau,\bv) \in \bH_\nu(\divvr;\O)\times \bH_0^1(\O).
\eeq

Based on the above discussion, one way to seek the robustness with respect to $\nu$ is using divergence-free finite element subspace of $V_{d0}$ and the formulation \eqref{epsbudiv0}, then the method is automatically $\nu$-robust. For a discussion of these finite elements, one can read the review paper \cite{John:17}. But some inf-sup stabilities are also needed for these divergence-free finite elements, thus careful analysis is still needed to ensure the methods are $\nu$-robust.  


Another choice is using stabilized methods. For a review of possible stabilized methods, one can read the corresponding Chapter in \cite{BBF:13}. In this paper, we use the augmented mixed methods as our method of choice.

\begin{rem} The situation is very similar to the diffusion equation with discontinuous coefficients. Consider 
\begin{equation}\label{scalar}
	-\nabla\cdot \,(\a(x)\nabla\, u) = f
 	\quad \mbox{in} \,\,\Omega
\end{equation}
with homogeneous Dirichlet boundary conditions (for simplicity) $u=0$ on $\p \O$; $f \in L^{2}(\O)$ is a given function; and diffusion coefficient $\a(x)$ is positive and piecewise constant with possible large jumps across subdomain boundaries (interfaces): $\a(x)=\a_i > 0$ in $\O_i$ for $i=1,\,...,\,n$. 
It is very natural that the norm for the robust analysis is 
\beq
\|\a^{1/2}\nabla v\|_0 \quad v\in H^1_0(\O).
\eeq
With this norm, we have the following robust coercivity (with a constant $1$) and continuity (with a constant $1$) :
$(\a\nabla v,\nabla v)  = \|\a^{1/2}\nabla v\|_0^2$  and $(\a\nabla w,\nabla v)  \leq \|\a^{1/2}\nabla w\|_0\|\a^{1/2}\nabla v\|_0 $, for $w,v\in H^1_0(\O)$. 
Let $V_c\subset H^1_0(\O)$ be a finite dimensional conforming approximation space and let $u_c \in V_c\subset H^1_0(\O)$ be the numerical solution of the problem $(\a\nabla u_c,\nabla v_c) = (f,v_c)$, for all $v_c\in V_c$. We have the robust best approximation result with respect to the energy norm: $\|\a^{1/2}\nabla(u-u_c)\|_0 = \inf_{v_c\in V_c} \|\a^{1/2}\nabla(u-v_c)\|_0$.
It is important to notice that the norm here we use is $\|\a^{1/2}\nabla v\|_0$, which is a semi-norm in $H^1(\O)$. It is a norm in $H^1_0(\O)$ (or any subspace of $H^1(\O)$ with a positive-measure homogeneous Dirichlet boundary $\Gamma_D$) because of the Poincar\'e-Friedrichs inequality: $\|v\|_0 \leq C\|\nabla v\|_0$ for $v\in H^1_0(\O)$. In most cases, the robust Poincar\'e-Friedrichs inequality
\beq
\|\a^{1/2}v\|_0 \leq C\|\a^{1/2} \nabla v\|_0 \quad v\in H^1_0(\O),
\eeq
with a constant $C$ independent of $\a$ is not true. Thus, only $\|\a^{1/2} \nabla v\|_0$ can be used, not its full $H^1$-version $(\|\a^{1/2}\nabla v\|_0^2 + \|\a^{1/2} v\|_0^2)^{1/2}$.

For the mixed method of \eqref{scalar}, let $\bsigma_m = -\a(x)\nabla u$, 
the mixed variational formulation is to find
$(\bsigma_m,\,u)\in H(\divvr;\O)\times L^2(\O)$ such that
\begin{equation}\label{mixed}
	\left\{\begin{array}{lclll}
 		(\a^{-1}\bsigma_m,\,\btau)-(\divv \btau,\, u)&=&0 \quad & \forall\,\, \btau \in H(\divvr;\O),\\[2mm]
		(\divv \bsigma_m, \,v) &=& (f,\,v)&\forall \,\, v\in L^2(\O).
	\end{array}\right.
\end{equation}
The energy norm of choice is $\|\a^{-1/2}\btau\|_0$, for $\btau \in H(\divvr;\O)$. The approximation of the divergence of $\bsigma_m$ is less important, since we already know its accurate value, which is $-f$. If an $H(\divvr;\O)$-type of norm is wanted, we can use the standard $\a$-weighted $H(\divvr)$-norm $\sqrt{\|\a^{-1/2}\btau\|_0^2+ \|\a^{-1/2}\gradt\btau\|_0^2}$ or a mesh and $\a$-weighted $H(\divvr)$-norm $\sqrt{\|\a^{-1/2}\btau\|_0^2+ \|\a^{-1/2}h\gradt\btau\|_0^2}$, see discussions in \cite{Zhang:20mixed}.

The above discussion suggests that our definition of energy norm in \eqref{norm_eg} is appropriate.
\end{rem}

\section{The augmented mixed formulations} \setcounter{equation}{0}

In this section, we present augmented mixed formulations with ultra-weak symmetry and robust coercivity and continuity.
\subsection{Ultra-weak mixed formulation}
Testing the second equation of \eqref{stress-velocity} with $\nu^{-1}\btau$ for $\btau\in \bH_\nu(\divvr;\O)$, we have
\beq \label{aug1}
(\nu^{-1}\cA\bsigma,\btau) - (\beps(\bu),\btau) = 0, \quad \forall \btau\in \bH_\nu(\divvr;\O).
\eeq
Define the 
$\o(\bbz) := (\nabla \bbz - (\nabla \bbz)^t)/2$ which is the anti-symmetry part of $\nabla \bbz$ for $\bbz\in H^1(\O)^d$.
Then it is true
$$
\nabla \bv = \beps(\bv) + \o(\bv).
$$
Using the integration by parts, we get
$$
(\beps(\bv),\btau) + (\o(\bv),\btau) + (\bv,\gradt\btau)  = (\nabla \bv,\btau) + (\bv,\gradt\btau) = 0,
\quad \btau \in \bH_{\nu}(\divvr;\O),\bv\in \bH_0^1(\O).
$$
Since the exact stress $\bsigma$ is symmetric, $\bsigma=\bsigma^t$, then
$$
(\bsigma, \o(\bv)) = \dfrac{1}{2}(\bsigma, \nabla \bv -( \nabla \bv )^t ) = 
\dfrac{1}{2}(\bsigma, \nabla \bv) -(\bsigma^t, (\nabla \bv )^t ) =0, \quad \forall \bv\in H^1(\O)^d.
$$
From the momentum equation, we have
$(\gradt\bsigma,\bv) = - (\bff,\bv)$, for all $\bv\in L^2(\O)^d$. 
Thus, we have
\beq \label{aug2}
(\gradt\bsigma,\bv) + (\bsigma, \o(\bv)) = - (\bff,\bv),\quad \bv\in \bH_0^1(\O).
\eeq
Combining \eqref{aug1} and \eqref{aug2}, we have the following ultra-weak mixed formulation:
find $(\bsigma,\bu)\in \bH_{\nu}(\divvr;\O)\times \bH_0^1(\O)$, such that
\begin{equation}\label{ultraweak}
\left\{
\begin{array}{lllll}
(\nu^{-1}\cA\bsigma,\btau) +b(\btau,\bu)  &=& 0,& \quad \forall \btau\in \bH_{\nu}(\divvr;\O),
 \\[2mm]
b(\bsigma,\bv) &=& - (\bff,\bv),&\quad \forall\bv\in \bH_0^1(\O),
\end{array}
\right.
\end{equation}
where the bilinear form $b$ has two equivalent formulations:
\beq
b(\btau,\bv): = (\gradt\btau,\bv) + (\btau, \o(\bv)) = - (\beps(\bv),\btau) \quad \forall \btau \in  \bH_{\nu}(\divvr;\O), \bv \in \bH_0^1(\O).
\eeq
Or, in a combined form, we have: find $(\bsigma,\bu)\in \bH_{\nu}(\divvr;\O)\times \bH_0^1(\O)$, such that
\beq \label{unltraweak_combined}
(\nu^{-1}\cA\bsigma,\btau) +b(\btau,\bu) - b(\bsigma,\bv) =  (\bff,\bv) \quad \forall (\btau,\bv) \in \bH_{\nu}(\divvr;\O)\times \bH_0^1(\O),
\eeq
or its symmetric variant
\beq \label{unltraweak_combined_sym}
(\nu^{-1}\cA\bsigma,\btau) +b(\btau,\bu) + b(\bsigma,\bv) =  -(\bff,\bv), \quad \forall (\btau,\bv) \in \bH_{\nu}(\divvr;\O)\times \bH_0^1(\O).
\eeq

\subsection{Augmented mixed formulation}
We propose two equivalent augmented mixed formulations based on the ultra-weak formulation, one symmetric and one non-symmetric.

Adding two consistent least-squares terms to \eqref{unltraweak_combined}:
\begin{eqnarray}\label{eq_aurterm1}
-\k(\nu^{-1}\cA\bsigma-\beps(\bu) , \cA\btau + \nu\beps(\bv)) &=& 0
~\mbox{ from the constitutive equation}, \\
\label{eq_augterm2}
\mbox{and  }
(\theta \nu^{-1}\gradt\bsigma,\gradt\btau) &=& -(\theta \nu^{-1}\bff,\gradt\btau)
~\mbox{ from the equilibrium equation},
\end{eqnarray}
we obtain the following problem: find $(\bsigma,\bu) \in \bH_{\nu}(\divvr;\O)\times \bH_0^1(\O)$, such that 
\beq\label{eq_bb}
B((\bsigma,\bu),(\btau,\bv))
= (\bff,\bv) - (\theta\nu^{-1}\bff,\gradt\btau)\quad \forall(\btau,\bv)\in \bH_{\nu}(\divvr;\O)\times \bH_0^1(\O),
\eeq
with the bilinear form $B$ defined as follows, for $(\bchi,\bw) \in  \bH_{\nu}(\divvr;\O)\times \bH_0^1(\O)$ and $(\btau,\bv) \in  \bH_{\nu}(\divvr;\O)\times \bH_0^1(\O)$,
\beq\notag
B((\bchi,\bw),(\btau,\bv)) :=(\nu^{-1}\cA\bchi,\btau) +b(\btau,\bw) - b(\bchi,\bv)-\k(\nu^{-1}\cA\bchi-\beps(\bw) , \cA\btau + \nu\beps(\bv)) + (\theta\nu^{-1}\gradt\bchi,\gradt\btau).
\eeq
Let $(\bchi,\bw) = (\btau,\bv)$ in $B((\bchi,\bw),(\btau,\bv))$ and use the fact $(\cA\btau,\btau)=(\cA\btau,\cA\btau)$, we get
\beq\notag
B((\btau,\bv),(\btau,\bv)) = (1-\kappa)\|\nu^{-1/2}\cA\btau\|^2_0 +\|\theta\nu^{-1/2}\gradt\btau\|^2_0+\k\|\nu^{1/2}\beps(\bv)\|_0^2.
\eeq
The number $1-\kappa$ should be positive to ensure coercivity. We set $\kappa=1/2$ for simplicity. Then \eqref{eq_bb} can be written as: find $(\bsigma,\bu) \in \bH_{\nu}(\divvr;\O)\times \bH_0^1(\O)
$, such that 
\beq\label{eq_varfrom}
B_\theta((\bsigma,\bu),(\btau,\bv)) = F_\theta(\btau,\bv)\quad\forall (\btau,\bv)\in \bH_{\nu}(\divvr;\O)\times \bH_0^1(\O),
\eeq
where, for $(\bchi,\bw) \in \bH_{\nu}(\divvr;\O)\times \bH_0^1(\O)$ and $(\btau,\bv) \in \bH_{\nu}(\divvr;\O)\times \bH_0^1(\O)$, the bilinear form $B_\theta$ and the linear form $F_\theta$ are defined as follows:
\begin{eqnarray*}
B_\theta((\bchi,\bw),(\btau,\bv))&:=&(\nu^{-1}\cA\bchi,\btau) + 2b(\btau,\bw) - 2b(\bchi,\bv) - (\cA\bchi, \beps(\bv))+ (\cA\btau, \beps(\bw)) \\
&&
+(\nu \beps(\bw),\beps(\bv))  + (\theta\nu^{-1}\gradt\bchi,\gradt\btau)
\\ 
&=&(\nu^{-1}\cA\bchi-\beps(\bw), \cA\btau-\nu\beps(\bv))- \dfrac{2}{d}(\gradt\bw, \tr \btau)
+ 2(\bchi,\beps(\bv))+ (\theta\nu^{-1}\gradt\bchi,\gradt\btau), \\ 
F_\theta(\btau,v) &=& 2(\bff,\bv)-(\theta\nu^{-1}\bff,\gradt\btau).
\end{eqnarray*}
Note that $B_\theta$ is not symmetric. We will give an equivalent symmetric version in Section \ref{sym_form}. Also, we notice that $B_\theta((\bchi,\bw),(\btau,\bv))$ is the standard least-squares bilinear form added with two extra terms $- \dfrac{2}{d}(\gradt\bw, \tr \btau)$ and $2(\bchi,\beps(\bv))$.
 We will discuss two choices of $\theta$, $\theta =1$ and a mesh dependent $\theta$, such that $\theta|_K = h^2_K$ for all $K\in\cT$. 

\begin{rem}
The ultra-weak mixed formulation and its augmented mixed methods presented here was originally proposed in \cite{AGR:20} for the linear elasticity and Boussinesq problem with temperature-dependent viscosity.
\end{rem}

\subsection{Some analysis for the augmented mixed formulations} \label{analysis}
Next, we discuss the coercivity and the continuity of the bilinear form $B$.
\begin{lem}{\bf (Robust Coercivity)}\label{lem_coea}
For $(\btau,\bv)\in \bH_{\nu}(\divvr;\O) \times \bH_0^1(\O)$, the following coercivity holds:
\begin{equation}\label{coercivity}
\begin{split}
B_{\theta}((\btau,\bv),(\btau,\bv)) =&\|\nu^{-1/2}\eps(\bv)\|^2_0 + \|\nu^{-1/2}\cA\btau\|_0^2 +  \|\theta\nu^{-1/2}\gradt\btau\|^2_0= \tri(\btau,v)\tri_{\eg,\theta}^2,
\end{split}
\end{equation}
\end{lem}
\begin{proof}
The coercivity \eqref{coercivity} follows immediately from the definition of the bilinear form $B_\theta$.
\end{proof}

\begin{lem}{\bf (Robust Continuity)}\label{lem_con}
 For $(\bchi,\bw),(\btau,\bv)\in \bH_{\nu}(\divvr;\O) \times \bH_0^1(\O)$, the following inequality is true:
\beq\label{ine_quasicon}
B_\theta((\bchi,\bw),(\btau,\bv))\leq 2\tri(\bchi,\bw)\tri_{\eg,\theta} \tri(\btau,\bv)\tri_{\full,\theta}.
\eeq
\end{lem}
\begin{proof}
For any $\btau\in \bH_{\nu}(\divvr;\O)$ and $\bv\in \bH_0^1(\O)$, it is easy to get that
$
\cA\btau - (\cA\btau)^t = \btau - \dfrac{1}{d} \tr(\btau)I - (\btau - \dfrac{1}{d} \tr(\btau)I )^t = \btau - \btau^t.
$
Thus, we have 
\begin{eqnarray*}
(\btau,\o(\bv)) &=& \frac12(\btau, \nabla \bv)-\frac12(\btau,(\nabla \bv)^t) =  \frac12(\btau-\btau^t, \nabla \bv) = \frac12(\cA\btau-(\cA\btau)^t, \nabla \bv) \\ 
&\leq& \frac12 \|\nu^{-1/2}(\cA\btau-(\cA\btau)^t)\|_0 \|\nu^{1/2}\nabla \bv\|_0 \leq \|\nu^{-1/2}\cA\btau\|_0 \|\nu^{1/2}\nabla \bv\|_0.
\end{eqnarray*}
Then, for the bilinear form $b(\btau,\bv)$ with $(\btau,\bv)\in \bH_{\nu}(\divvr;\O)\times \bH_0^1(\O)$, we have the following two bounds.
\begin{eqnarray} \label{b_bound1}
b(\btau,\bv) &=& - (\beps(\bv),\btau) \leq \|\nu^{1/2}\beps(\bv)\|_0\|\nu^{-1/2}\btau\|_0,\\[2mm] \label{b_bound2}
b(\btau,\bv) &=&  (\gradt\btau,\bv)+(\btau,\o(\bv)) \leq \|\nu^{-1/2}\theta^{1/2}\gradt\btau\|_0\|\nu^{1/2}\theta^{-1/2}\bv\|_0+\|\nu^{-1/2} \cA\btau\|_0 \|\nu^{1/2}\nabla \bv\|_0.
\end{eqnarray}
By the  Cauchy-Schwartz inequality,  \eqref{ine_cA}, \eqref{b_bound1}, and \eqref{b_bound2}, we have
\begin{equation*}
\begin{split}
B_\theta((\bchi,\bw),(\btau,\bv))&=(\nu^{-1}\cA\bchi,\btau) + 2b(\btau,\bw) - 2b(\bchi,\bv) -(\cA\bchi, \beps(\bv))+ (\cA\btau, \beps(\bw)) \\
&
+(\nu \beps(\bw),\beps(\bv))  + (\theta\nu^{-1}\gradt\bchi,\gradt\btau)\\
\leq& \|\nu^{-1/2}\cA\bchi\|_0\|\nu^{-1/2}\cA\btau\|_0 + 2\|\nu^{1/2}\beps(\bw)\|_0\|\nu^{-1/2}\btau\|_0
+ 2 \|\nu^{-1/2}\theta^{1/2}\gradt\bchi\|_0\|\nu^{1/2}\theta^{-1/2}\bv\|_0 \\
&+2\|\nu^{-1/2} \cA\bchi\|_0 \|\nu^{1/2}\nabla \bv\|_0 + \|\nu^{-1/2}\cA\bchi\|_0 \|\nu^{1/2}\beps(\bv)\|_0 + \|\nu^{1/2}\beps(\bw)\|_0 \|\nu^{-1/2}\cA\btau\|_0  \\
&+ \|\nu^{1/2} \beps(\bw)\|_0\|\nu^{1/2}\beps(\bv)\|_0  + \|\theta^{1/2}\nu^{-1/2}\gradt\bchi\|_0 \|\theta^{1/2}\nu^{-1/2}\gradt\btau\|_0
\\
\leq& 2\tri(\bchi,\bw)\tri_{\eg,\theta} \tri(\btau,\bv)\tri_{\full,\theta}.
\end{split}
\end{equation*}
The lemma is proved.
\end{proof}

\subsection{A symmetric formulation} \label{sym_form}
The formulation \eqref{eq_varfrom} is not symmetric. In many situations, for example, eigenvalues problems or developing efficient linear solvers, symmetric formulations are always preferred. Also, we can always associate a Ritz-minimization variational principle to a symmetric problem. Luckily, the method  \eqref{eq_varfrom}  is equivalent to a symmetric GLS formulation by adding least-squares residuals
$$
-\frac12(\nu^{-1}\cA\bsigma-\beps(\bu),\cA\btau - \nu\beps(\bv)) = 0 \quad\mbox{and}\quad
\frac12(\nu^{-1}\gradt\bsigma,\gradt\btau) = -\frac12(\nu^{-1}\bff,\gradt\btau),
$$
to the  symmetric saddle point ultra-weak formulation \eqref{unltraweak_combined_sym}. We have the following symmetric formulation: find $(\bsigma, u) \in \bH_{\nu}(\divvr;\O) \times \bH_0^1(\O)$, such that
\beq \label{eq_varfrom_sym}
B_{sym,\theta}((\bsigma, \bu),(\btau,\bv))= F_{sym,\theta}(\btau,\bv), \quad \forall (\btau, \bv) \in \bH_{\nu}(\divvr;\O) \times \bH_0^1(\O),
\eeq
with the forms are defined for $(\bchi,\bw) \in \bH_{\nu}(\divvr;\O) \times \bH_0^1(\O)$ and $(\btau,\bv) \in \bH_{\nu}(\divvr;\O) \times \bH_0^1(\O)$,
\begin{eqnarray*}
B_{sym,\theta}((\bchi,\bw),(\btau,\bv))&:=&(\nu^{-1}\cA\bchi,\btau) + 2b(\btau,\bw) + 2b(\bchi,\bv) + (\cA\bchi, \beps(\bv))+ (\cA\btau, \beps(\bw)) 
\\ \nonumber
&&
-(\nu \beps(\bw),\beps(\bv))  + (\theta\nu^{-1}\gradt\bchi,\gradt\btau) \\ 
F_{sym,\theta}(\btau,v) &:=& -2(\bff,\bv)-(\theta\nu^{-1}\bff,\gradt\btau).
\end{eqnarray*}
It is easy to see that \eqref{eq_varfrom} and  \eqref{eq_varfrom_sym} are equivalent since replacing the test function $\bv$ by $-\bv$ in one formulation leads to the other formulation. By doing so, we know that \eqref{eq_varfrom} and  \eqref{eq_varfrom_sym} and their corresponding finite element formulations produce identical solutions. Thus all the analysis of the non-symmetric formulations can be applied to the symmetric versions.

Alternatively, we can establish the inf-sup stability of the symmetric formulation directly.
\begin{lem}
The following robust inf-sup stability with the stability constant being $1$ holds:
\beq \label{infsup}
\sup_{(\btau,\bv)\in \bH_{\nu}(\divvr;\O) \times \bH_0^1(\O)} \dfrac{B_{sym,\theta}((\bchi,\bw),(\btau,\bv))}{\tri (\bchi,w) \tri_{\eg,\theta}  \tri (\btau,\bv) \tri_{\eg,\theta} } \geq  \tri (\bchi,\bw) \tri_{\eg,\theta}
\quad \forall (\bchi,\bw) \in \bH_{\nu}(\divvr;\O) \times \bH_0^1(\O).
\eeq
\end{lem}
\begin{proof} 
The lemma can be easily proved by choosing $(\btau,\bv) =  (\bchi,-\bw)$ and using the facts that $B_{\theta}((\bchi,\bw),(\bchi,\bw) = B_{sym,\theta}((\bchi,\bw),(\bchi,-\bw))$ and \eqref{coercivity}. 
\end{proof}

\section{Augmented mixed finite element approximations}
\setcounter{equation}{0}

Let $\bSigma_h \subset \bH_{\nu}(\divvr;\O)$ and $\bS_h \subset \bH_0^1(\O)$ be the finite element spaces. we have the following discrete augmented mixed finite element problem:  find $(\bsigma_h,\bu_h) \in \bSigma_h\times \bS_h$, such that 
\beq\label{eq_varfrom_dis}
B_\theta((\bsigma_h,\bu_h),(\btau_h,\bv_h)) = F_\theta(\btau_h,\bv_h)\quad\forall (\btau_h,\bv_h)\in \bSigma_h\times \bS_h.
\eeq
Due to that fact that the finite element spaces are conforming, the discrete problem \eqref{eq_varfrom_dis} is well-posed. And we have the following error equation.
\beq \label{err_eq}
B_\theta((\bsigma-\bsigma_h,\bu-\bu_h),(\btau_h,\bv_h)) = 0\quad\forall (\btau_h,\bv_h)\in \bSigma_h\times \bS_h.
\eeq
We immediately have the following robust Cea's Lemma type of result and local robust and optimal convergence results for each choices. 
\begin{thm} \label{Cea_thm}
Let $(\bsigma,\bu)$ be the solution of \eqref{eq_varfrom} and  $(\bsigma_h,\bu_h)\in \bSigma_h\times \bS_h$ be the finite element solution of \eqref{eq_varfrom_dis}, we have
\beq \label{Cea}
\tri (\bsigma-\bsigma_h,\bu-\bu_h) \tri_{\eg,\theta} \leq 2 \inf_{(\btau_h,\bv_h)\in \bSigma_h\times \bS_h} \tri (\bsigma-\btau_h,\bu-\bv_h) \tri_{\full,\theta}
\eeq
Assume that $(\bsigma_{rt,1,h},\bu_{1,1,h})\in \bRT_{0,\nu}\times S_{1,0}^d$ is the solution of problem \eqref{eq_varfrom_dis} with $\theta=1$ , and $\bSigma_h\times \bS_h =\bRT_{0,\nu}\times S_{1,0}^d$.  Assume that $\bu|_K\in H^{1+s_K}(K)^d$, $\bsigma|_K\in H^{q_K}(K)^{d\times d}$, $\bff|_K\in H^{t_K}(K)^d$ for $K\in\cT$, where the local regularity indexes $s_K$, $q_K$, and $t_K$ satisfy the following assumptions: $0<s_K\leq 1$ in two dimensions and $1/2<s_K\leq 1$ in three dimensions,  $1/2<q_K\leq 1$ with the constant $C_{rt}>0$ being unbounded as $q_K\downarrow 1/2$, and $0<t_K\leq 1$, then the following robust and local optimal a priori error estimate holds: there exists a constant $C$ independent of $\nu$ and the mesh-size, such that
\beq
\tri(\bsigma- \bsigma_{rt,1,h},\bu-\bu_{1,1,h})\tri_{\eg,1}
 \leq C\sum_{K\in\cT}\nu_K^{1/2}\left(h_K^{s_K}|\nabla \bu|_{s_K,K}+ \nu_K^{-1}(C_{rt} h_K^{q_K}|\bsigma|_{q_K,K} + h_K^{t_K}|\bff|_{t_K,K})
\right).
\eeq
Assume that $(\bsigma_{rt,h,h},\bu_{1,h,h})\in \bRT_{0,\nu}\times S_{1,0}^d$ is the solution of problem \eqref{eq_varfrom_dis} with $\theta|_K=h_K^2$, $K\in\cT$, and $\bSigma_h\times \bS_h =\bRT_{0,\nu}\times S_{1,0}^d$. 
Assume that $u|_K\in H^{1+s_K}(K)^d$, $\bsigma|_K\in H^{q_K}(K)^{d\times d}$, $\bff|_K\in H^{t_K}(K)^d$ for $K\in\cT$, where the local regularity indexes $s_K$, $q_K$, and $t_K$ satisfy the following assumptions: $0<s_K\leq 1$ in two dimensions and $1/2<s_K\leq 1$ in there dimensions,  $1/2<q_K\leq 1$ with the constant $C_{rt}>0$ being unbounded as $q_K\downarrow 1/2$, and $0<t_K\leq 1$, then the following robust and local optimal a priori error estimate holds: there exists a constant $C$ independent of $\nu$ and the mesh-size, such that
\begin{eqnarray} 
\tri(\bsigma- \bsigma_{rt,h,h},\bu-\bu_{1,h,h})\tri_{\eg,h}
 &\leq& C\sum_{K\in\cT}\nu_K^{1/2}\left(h_K^{s_K}|\nabla \bu|_{s_K,K}+ \nu_K^{-1}(C_{rt} h_K^{q_K}|\bsigma|_{q_K,K} + h_K^{1+t_K}|\bff|_{t_K,K})
\right).
\end{eqnarray}
Assume that $(\bsigma_{bdm,h,h},\bu_{2,h,h})\in \bBDM_{0,\nu}\times S_{2,0}^d$ is the solution of problem \eqref{eq_varfrom_dis} with $\theta|_K=h_K^2$ and $K\in\cT$ and $\bSigma_h\times \bS_h =\bBDM_{0,\nu}\times S_{2,0}^d$. 
Assume that $\bu\in (H^3(\O))^d$, $\bsigma|_K\in H^2(K)^{d\times d}$, $\bff|_K\in H^1(K)^d$ for $K\in\cT$, then the following robust and local optimal a priori error estimate holds: there exists a constant $C$ independent of $\nu$ and the mesh-size, such that
\begin{eqnarray} 
\tri(\bsigma- \bsigma_{bdm,h,h},\bu-\bu_{2,h,h})\tri_{\eg,h}
 &\leq& C\sum_{K\in\cT}\nu_K^{1/2}h_K^2 \left(|\nabla \bu|_{2,K}+ \nu_K^{-1}(|\bsigma|_{2,K} + |\bff|_{1,K})
\right).
\end{eqnarray}

\end{thm}
\begin{proof}
The proof is quite standard with the help of error equation \eqref{err_eq}, robust coercivity \eqref{coercivity}, and robust continuity \eqref{ine_quasicon}.
\begin{eqnarray*}
\tri(\bsigma-\bsigma_h,\bu-\bu_h)\tri_{\eg,\theta}^2 &=& B_{\theta}((\bsigma-\bsigma_h,\bu-\bu_h),(\bsigma-\bsigma_h,\bu-\bu_h))= B_{\theta}(\bsigma-\bsigma_h,u-u_h;\bsigma-\btau_h,\bu-\bv_h)\\
&\leq&  2\tri(\bsigma-\bsigma_h,\bu-\bu_h)\tri_{\eg,\theta}\tri(\bsigma-\btau_h,\bu-\bv_h)\tri_{\full,\theta}.
\end{eqnarray*}

The individual convergence results can be derived from regularity assumptions and interpolation results \eqref{nodal_inter}, \eqref{RT_inter1}, and \eqref{RT_inter2}.
\end{proof}

\begin{rem}
Note that the divergence $\gradt\bsigma = -\bff$ is a known quantity, thus in principle we should not let the poor regularity and approximation of the divergence of $\bsigma$ to ruin the convergence order of $\bu$ and $\bsigma$.
The mesh-weighted augmented mixed formulation is useful when the regularity of $\bff$ is low for the RT case. For the BDM case, since the divergence approximation of BDM is one-order lower, this will help the formulation to achieve optimal convergence. 

\end{rem}

\section{Related Least-Squares Formulations and A Posteriori Error Analysis}
\setcounter{equation}{0}

In this section, we discuss the related least-squares formulations and  a posteriori error analysis.
For the first-order system stress-velocity formulation \eqref{stress-velocity}, the least-squares functional is
\beq \label{lsfunctional}
J_{\theta}(\btau,\bv;\bff) := \|\nu^{1/2}\beps(\bv) - \nu^{-1/2}\cA\btau\|_0^2+  \|\theta^{1/2}\nu^{-1/2}(\gradt\btau +\bff)\|_0^2 \quad (\btau,\bv)\in \bH_{\nu}(\divvr;\O) \times \bH_0^1(\O).
\eeq
Then the $L^2$-based least-squares minimization problem is: find $(\btau,\bv)\in \bH_{\nu}(\divvr;\O) \times \bH_0^1(\O)$, such that 
\beq
J_{\theta}(\bsigma,\bu;\bff) = \inf_{(\btau,\bv)\in \bH_{\nu}(\divvr;\O) \times \bH_0^1(\O)} J_{\theta}(\btau,\bv;\bff) .
\eeq 
The least-squares finite element method (LSFEM) is: find $(\bsigma_{h}^{ls},\bu_h^{ls}) \in \bSigma_h\times \bS_h$, such that
\beq \label{lsfem1}
J_{\theta}(\bsigma^{ls}_h,\bu^{ls}_h;\bff) = \inf_{(\btau,\bv)\in \in \bSigma_h\times \bS_h.} J_{\theta}(\btau,\bv;\bff).
\eeq 
For the case that $\theta=1$, the following norm-equivalence is established in \cite{CLW:04} (without $\nu$-weighting):
\beq \label{ls_equvalience}
C_{coe}(\nu) \tri(\btau,\bv)\tri_{\eg,1} ^2 \leq J_\theta(\btau,\bv;\bzero) \leq C_{con} \tri(\btau,\bv)\tri_{\eg,1} ^2, \quad \forall  (\btau,\bv) \in \bH_{\nu}(\divvr;\O) \times \bH_0^1(\O).
\eeq
The upper bound in \eqref{ls_equvalience} can be easily proved by simple triangle inequalities. The proof of the lower bound in \eqref{ls_equvalience} relies on the Korn's first inequality. Since the $\nu$-robust Korn's first inequality is not available, the constant  $C_{coe}(\nu)$ may be $\nu$-dependent. Thus there is no robust a priori error analysis of the corresponding LSFEM with $\theta=1$. 
For the case $\theta|_K = h_K^2$, for $K\in\cT$, even a $\nu$-dependent  norm equivalence similar to \eqref{ls_equvalience} is not available, thus there is no (even $\nu$-dependent) a priori error analysis of the corresponding LSFEM.

As discussed in \cite{LZ:24}, the least-squares functional $J_{\theta}$ can be used as a simple a posteriori error estimator. Let  $(\bsigma_h,\bu_h)\in \bSigma_h\times \bS_h$ be the finite element solution of \eqref{eq_varfrom_dis},  we  define the following least-squares functional a posteriori error estimator:
\beq\label{ap_lsfem}
\eta_{\theta}(\bsigma_h,\bu_h)
= \Big (\|\nu^{1/2}\beps(\bu_h) +\nu^{-1/2}\cA\bsigma_h\|_0^2+  \|\theta^{1/2}\nu^{-1/2}(\gradt\bsigma_h +\bff)\|_0^2
\Big )^{1/2} = J_\theta(\bsigma_{h},\bu_{h};\bff)^{1/2}.
\eeq
The local indicator is defined as 
\beq
\eta_{\theta,K}(\bsigma_h,\bu_h)
= \Big (\|\nu^{1/2}\beps(\bu_h) +\nu^{-1/2}\cA\bsigma_h\|_{0,K}^2+  \|\theta^{1/2}\nu^{-1/2}(\gradt\bsigma_h +\bff)\|_{0,K}^2
\Big )^{1/2}.
\eeq
Due to that \eqref{ls_equvalience} is true for any $(\btau,\bv) \in \bH_{\nu}(\divvr;\O) \times \bH_0^1(\O)$, we have the following (non-$\nu$-robust) reliability and efficiency for the case $\theta=1$.
\beq \label{esti}
C_{coe}(\nu) \tri(\bsigma-\bsigma_h,\bu-\bu_h)\tri_{\eg,1} ^2 \leq J_1(\bsigma_h,\bu_h;\bzero) = \eta_{1}(\bsigma_h,\bu_h)^2 \leq C_{con} \tri(\bsigma-\bsigma_h,\bu-\bu_h)\tri_{\eg,1} ^2.
\eeq
In \cite{LZ:24}, for the diffusion equation with discontinuous coefficient, we proved that the corresponding  least-squares functional error estimator is robust for both cases ($\theta=1$ or $\theta=h_K^2$). The situation is more complicated here since the we have two norms $\tri\cdot\tri_{\eg}$ and  $\tri\cdot\tri_{\full}$. At current stage, there is no robust analysis of these error estimators. The numerical experiments in Section 8 show that they are computationally $\nu$-robust.

Since the augmented mixed method uses the same conforming finite element spaces as the standard least-squares finite element methods, the least-squares functional error estimator can also be used for other augmented mixed methods previously developed, for example,  for Stokes equations in \cite{Gatica:06,FGM:08}.

\section{Numerical Tests}
\setcounter{equation}{0}
%

\subsection{Kellogg-type example with exact solutions}\label{kellogg}
We first construct solutions with singularity in the sprit of the Kellogg's example \cite{Kellogg:74}.
 
Consider $\overline{\O} = [-1,1]\times[-1,1]$ and $\O = \cup_{i=1}^4\O_i$ with $\O_i$ being the domain in $i$-th quadrant. We assume that $\nu_i = \nu|_{\O_i}>0$ is a constant in each subdomain. 
We construct an example for the following system with inhomogeneous boundary condition $\bu = \bg$ on $\partial\O$:
$$
\left\{
\begin{array}{lllll}
- \gradt(\nu \beps(\bu)) +\nabla p &=& \bzero  & \mbox{in } \O,
 \\[1mm]
 \gradt \bu &=& 0 & \mbox{in } \O,
\end{array}
\right.
\quad\mbox{or}\quad 
\left\{
\begin{array}{lllll}
\gradt\bsigma    & =& \bzero & \mbox{in } \O,
 \\[1mm]
 \bsigma - \nu\beps(\bu) + pI & =& 0 & \mbox{in } \O,
 \\[1mm]
 \gradt \bu &=& 0 & \mbox{in } \O.
\end{array}
\right.
$$
The inhomogeneous boundary condition $\bg$ is given by the exact solution (other types of boundary conditions can also be given with the available exact solution).

To construct a divergence-free solution, we use the Neuber-type solution \cite{TB:82}: 
$$
\bu|_{\O_i} = \grad(\bx\cdot{\bf B}) - 2{\bf B} \quad\mbox{and}\quad
p|_{\O_i} = \nu\gradt{\bf B} \quad \mbox{in  } \O_i, i=1,\cdots,4,
$$ 
where $\bx = (x,y)^t$. The vector function ${\bf B} = (B_1,B_2)^t$ is a harmonic vector function in each $\O_i$, that is,
$\Delta{\bf B}|_{\O_i} = (\Delta B_1|_{\O_i}, \Delta B_2|_{\O_i})^t ={\bf 0}$. Note that $B_1$ and $B_2$ are harmonic in each subdomain $\O_i$ but not in $\O.$   Then in each subdomain $\O_i$, we have 
$$
\gradt(\nu_i\eps(\bu) - pI) =  \gradt(\frac{\nu_i}2\grad\bu) - pI ) =0
\quad\mbox{and}\quad
\gradt\bu =0  \quad\mbox{in}\quad \O_i, \quad i=1,\cdots,4.
$$
As discussed in Chapter 13 of \cite{Zill:DE}, a choice for a harmonic $\bf B$ in the polar coordinate $(r,\vartheta)$ is
\beq \label{BB}
B_1|_{\O_i} = r^\a(\ma_i\sin(\a\vartheta) + \mb_i\cos(\a\vartheta) ) 
\quad\mbox{and}\quad B_2|_{\O_i} = r^\a(\mc_i\sin(\a\vartheta) + \md_i\cos(\a\vartheta) )
\quad i=1,\cdots,4.
\eeq
where $\a$, $\ma_i$, $\mb_i$, $\mc_i$, and $\md_i$ ($i=1,\cdots,4$) are coefficients to be determined.
Some simple computations (see detailed calculations in Appendix) give
\begin{eqnarray}
u_1|_{\O_i} &=& \a r^\a[(\ma_i\cos\vartheta + \mc_i\sin\vartheta)\sin(\a-1)\vartheta
 + (\mb_i\cos\vartheta+\md_i\sin\vartheta)\cos(\a-1)\vartheta] - B_1|_{\O_i},\\
 u_2|_{\O_i} &=& \a r^\a[(\ma_i\cos\vartheta + \mc_i\sin\vartheta)\cos(\a-1)\vartheta
 - (\mb_i\cos\vartheta+\md_i\sin\vartheta)\sin(\a-1)\vartheta] - B_2|_{\O_i},\\
 p|_{\O_i} &=& \nu_i\a r^{\a-1}[ (\ma_i-\md_i)\sin(\a-1)\vartheta + (\mb_i+\mc_i)\cos(\a-1)\vartheta ].
\end{eqnarray}
Thus, we have $\a$, $\nu_i$,  $\ma_i,\mb_i,\mc_i$ and $\md_i$ for $i=1,\cdots,4$, a total of 21 unknowns to be determined.

We need interface conditions to ensure that $\bsigma = \nu\eps(\bu) - pI \in \bH(\divvr;\O)$ and $\bu$ in $\bH^1(\O)$:
$$
\jump{( \nu\eps(\bu) - pI )\bn_{ \G_{ij} } }_{ \G_{ij} } = 0
\quad\mbox{and}\quad
\jump{\bu}_{ \G_{ij} } = 0 \quad \mbox{on}\quad \Gamma_{ij} =\O_i\cap\O_j.
$$
There are $4$ equations on each interface, thus there are 16 equations. The detailed derivation of the 16 equations \eqref{eq01}-\eqref{eq16} can be found in the appendix.

Let $\md_4=1$, $\nu_2=\nu_4=1$, and $\nu_1=\nu_3$. Then for a given $\alpha \in (0,1]$, we have 16 equations \eqref{eq01}-\eqref{eq16} and 16 unknowns ( $\nu_1$ and the rest of $\ma_i$, $\mb_i$, $\mc_i$, $\md_i$). For a series of different $\alpha$, we solve the nonlinear system by $\mathtt{fslove}$ in Matlab.
%
The following tables Tables \ref{Tab1_numbers}-\ref{Tab3_numbers} present some data of the solutions.
\begin{table}[h]
\caption{Data 1(left) $\nu_1=\nu_3 = 160.3374$ 
and $\a = 0.13.$
Data 2 (right) $\nu_1=\nu_3 = 67.1849$ and $\a = 0.2.$}
\label{Tab1_numbers}
\begin{tabular}{|c|c|c|c|c|}
\hline
  $i$  & $\ma_i$    &  $\mb_i$   &  $\mc_i$  &    $\md_i$      \\\hline
 $1$ &$0.0132$ & $0.3067$    & $-0.0482$  &  $-0.2740$   \\\hline
 $2$  &$-2.0673$  & $0.5604$   & $1.2592$  &  $-0.5447$ \\\hline
 $3$  &$-0.1340$ & $-0.2763$   & $0.1530$  &  $0.2323$   \\\hline
 $4$  &$1.6747$   & $-1.3353$    & $-0.9393$  &   $1.0$   \\\hline
\end{tabular}
\quad
\begin{tabular}{|c|c|c|c|c|}
\hline
  $i$  & $\ma_i$    &  $\mb_i$   &  $\mc_i$  &    $\md_i$      \\\hline
 $1$ &$0.0134$ & $0.2523$    & $-0.0657$  &  $-0.2527$   \\\hline
 $2$  &$-0.8757$  & $0.3244$   & $0.9149$  &  $-0.5713$ \\\hline
 $3$  &$-0.1591$ & $-0.1963$   & $0.2017$  &  $0.1658$   \\\hline
 $4$  &$0.5178$   & $-0.7772$    & $-0.4044$  &   $1.0$   \\\hline
\end{tabular}
\end{table}


\begin{table}[h]
\caption{Data 3 (left) $\nu_1=\nu_3 = 29.3162$ and $\a = 0.3.$
Data 4 (right)  $\nu_1=\nu_3 = 16.0517$ and $\a = 0.4$
}\label{Tab2_numbers}
\begin{tabular}{|c|c|c|c|c|}
\hline
  $i$  & $\ma_i$    &  $\mb_i$   &  $\mc_i$  &    $\md_i$      \\\hline
 $1$ &$0.0179$ & $0.2853$    & $-0.1169$  &  $-0.3016$   \\\hline
 $2$  &$-0.5390$  & $0.2564$   & $0.7106$  &  $-0.7233$ \\\hline
 $3$  &$-0.2414$ & $-0.1532$   & $0.3127$  &  $0.0827$   \\\hline
 $4$  &$0.1094$   & $-0.5867$    & $0.1675$  &   $1.0$   \\\hline
\end{tabular}
\quad
\begin{tabular}{|c|c|c|c|c|}
\hline
  $i$  & $\ma_i$    &  $\mb_i$   &  $\mc_i$  &    $\md_i$      \\\hline
 $1$ &$0.0434$ & $0.5249$    & $-0.2808$  &  $-0.5181$   \\\hline
 $2$  &$-0.7143$  & $0.4998$   & $0.6608$  &  $-1.2022$ \\\hline
 $3$  &$-0.5126$ & $-0.1209$   & $0.5795$  &  $-0.1070$   \\\hline
 $4$  &$-0.2546$   & $-0.8338$    & $0.9392$  &   $1.0$   \\\hline
\end{tabular}
\end{table}


\begin{table}[h]
\caption{Data 5 of the solution with  $\nu_1=\nu_3 = 9.8990$ and $\a = 0.5.$}\label{Tab3_numbers}
\begin{tabular}{|c|c|c|c|c|}
\hline
  $i$  & $\ma_i$    &  $\mb_i$   &  $\mc_i$  &    $\md_i$      \\\hline
 $1$ &$0.2364$ & $2.2978$    & $-1.4918$  &  $-2.0518$   \\\hline
 $2$  &$-2.2978$  & $2.3401$   & $1.0000$  &  $-4.5437$ \\\hline
 $3$  &$-2.2978$ & $0.2364$   & $2.0518$  &  $-1.4918$   \\\hline
 $4$  &$-2.3401$   & $-2.2978$    & $4.5437$  &   $1.0$   \\\hline
\end{tabular}
\end{table}

%
%

The exact solutions $\bu = (u_1 , u_2)^t$, with the Data1 and Data 5, are presented in Figure \ref{fig_Data1_exactu} and Figure \ref{fig_Data5_exactu}, respectively. The solutions behave like the classical Kellogg's solution \cite{Kellogg:74,CZ:09} for the diffusion problem. 
\begin{figure}[htbp]
  \centering
  \begin{minipage}[t]{0.45\linewidth}
  \includegraphics[scale=0.5]{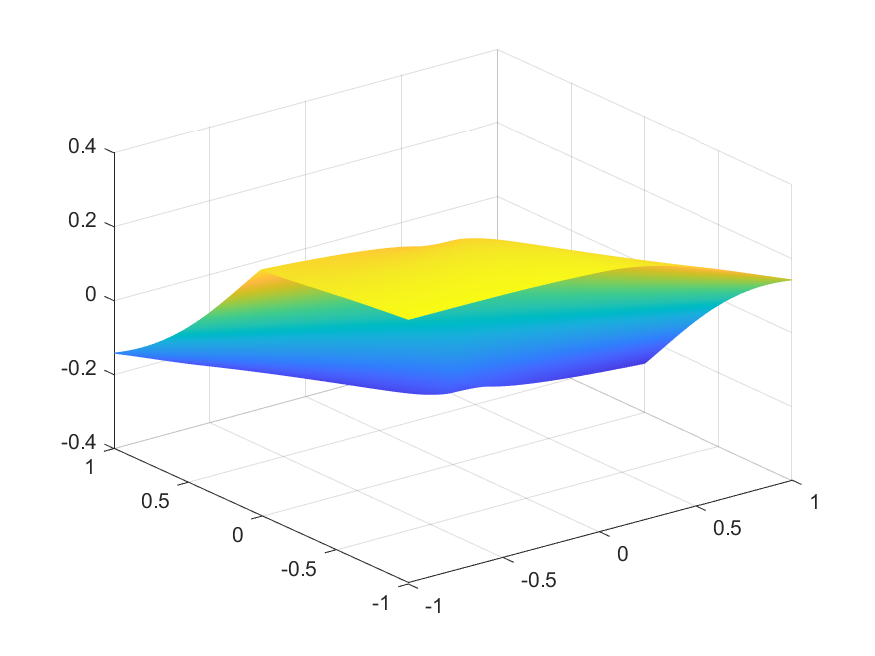}   
  \end{minipage}
    \begin{minipage}[t]{0.45\linewidth}
  \includegraphics[scale=0.5]{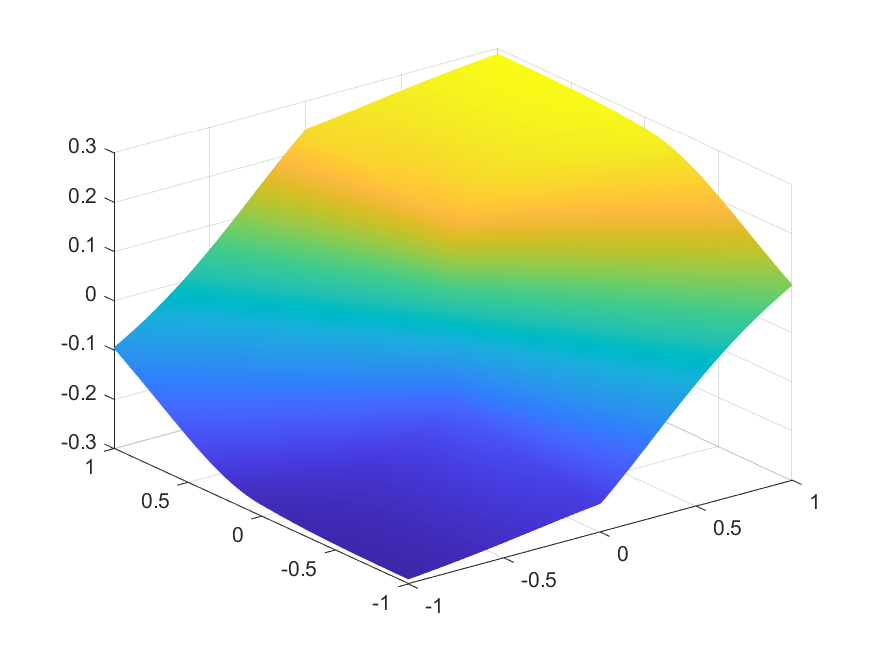}   
  \end{minipage}
  \caption{Exact solutions $u_1$ (left) and $u_2$ (right) for Data1 with $\a = 0.13$.}
    \label{fig_Data1_exactu}
\end{figure}

\begin{figure}[htbp]
  \centering
  \begin{minipage}[t]{0.45\linewidth}
  \includegraphics[scale=0.5]{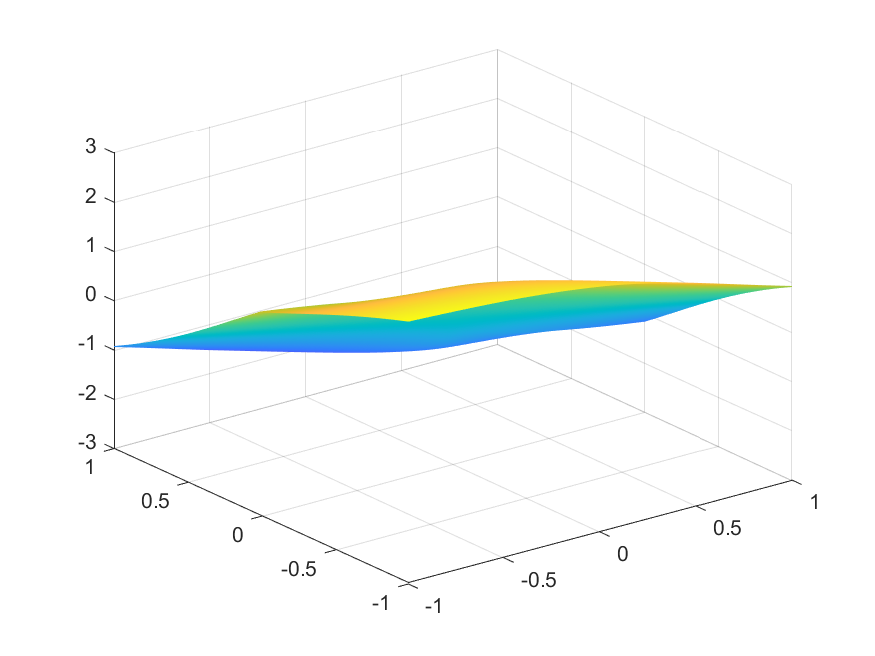}   
  \end{minipage}
    \begin{minipage}[t]{0.45\linewidth}
  \includegraphics[scale=0.5]{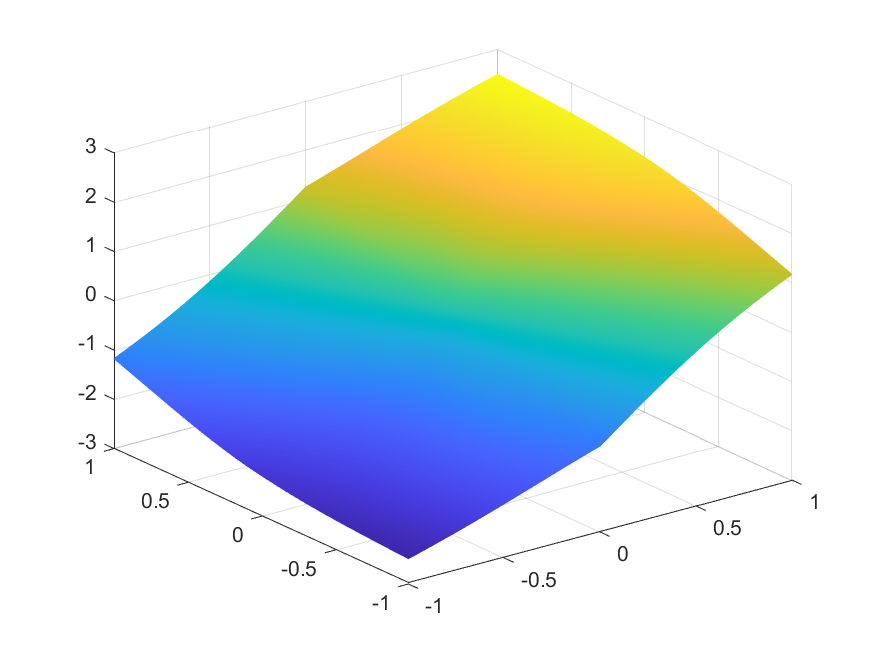}   
  \end{minipage}
  \caption{Exact solutions $u_1$ (left) and $u_2$ (right) for Data5 with $\a = 0.5$.}
    \label{fig_Data5_exactu}
\end{figure}

\subsection{Robustness of the a prior estimate}
To check the robustness of a priori error estimate, let $I^{rt}_{h,\nu}$ and $I^{bdm}_{h,\nu}$ be the interpolations defined in \eqref{Irtbdm}, and let $I^{l2}_h$ to be the $L^2$ projection from $\bH^1(\O)$ to $S_{1,0}^d$ if $\bRT_{0,\nu}\times S_{1,0}^d$ is chosen or to $S_{2,0}^d$ if $\bBDM_{1,\nu}\times S_{2,0}^d$ is used in the test, respectively.

The index in the a prior estimate is computed by 
\beq\notag
\mbox{ind-err} := \frac{\tri (\bsigma-\bsigma_h,\bu - \bu_h )\tri_{\eg,\theta}}
{\tri (\bsigma- I^{rt}_{h,\nu}\bsigma,\bu - I^{l2}_h\bu )\tri_{\full,\theta}}
\quad \mbox{or} \quad
\mbox{ind-err} := \frac{\tri (\bsigma-\bsigma_h,\bu - \bu_h )\tri_{\eg,\theta}}
{\tri (\bsigma- I^{bdm}_{h,\nu}\bsigma,\bu - I^{l2}_h\bu )\tri_{\full,\theta}}
\eeq
for $\bRT_{0,\nu}\times S_{1,0}^d$ or $\bBDM_{1,\nu}\times S_{2,0}^d$ respectively.
If the method is robust, the number $\mbox{ind-err}$ should be a number smaller than $1$ and of similar size for different $\nu$, .

Tables \ref{tab_AugRT_compare}-\ref{tab_Augbdm_compare} show the index of our proposed methods for Data 1-Data 5 with the stopping criteria $\mbox{rel-err}<0.11$ for the case using $\bRT_{0,\nu}\times S^d_{1,0}$, and $\mbox{rel-err}<0.05$ for the case of $\bBDM_{1,\nu}\times S^d_{2,0}$. From the tables, we observe that as $\nu$ becomes larger, the regularity index $\a$ becomes smaller, but the indexes of a prior estimate are of similar size. This confirms that the proposed augmented mixed method is indeed robust with respect to $\nu$.

\begin{table}[htbp]
\caption{Index of the a prior estimate, number of refinements $k$, number of elements $n$, the final 
$\tri (\bsigma-\bsigma_h,\bu - \bu_h )\tri_{\eg,\theta}$,  
and the final interpolation error 
$\tri (\bsigma- I^{rt}_{h,\nu}\bsigma,
\bu - I^{l2}_h\bu )\tri_{\full,\theta}$
with  $\bRT_{0,\nu}\times S^d_{1,0}$ and $\theta_K = 1$.}
\label{tab_AugRT_compare}
\begin{tabular}{|c|c|c|c|c|c|c|}
\hline
&$\mbox{ind-err}$& $k$  & $n$  & 
$ \tri (\bsigma-\bsigma_h,\bu - \bu_h )\tri_{\eg,\theta}$ &
$\tri (\bsigma- I^{rt}_{h,\nu}\bsigma,
\bu - I^{l2}_h\bu )\tri_{\full,\theta}$  \\\hline
$\mbox{Data}1$ & $0.9855$& $49$  &  $4416$ & $0.4914$ & $0.4986$  \\\hline
$\mbox{Data}2$ & $0.9558$& $39$  &  $4548$ & $0.2279$ & $0.2384$ \\\hline
$\mbox{Data}3$ & $0.9437$& $27$  &  $3318$ & $0.1677$ & $0.1777$ \\\hline
$\mbox{Data}4$ & $0.9242$& $19$  &  $2064$ & $0.2285$ & $0.2472$ \\\hline
$\mbox{Data}5$ & $0.9239$& $14$  &  $1296$ & $0.7443$ & $0.7750$ \\\hline
\end{tabular}
\end{table}

\begin{table}[htbp]
\caption{Index of the a prior estimate, number of refinements $k$, number of elements $n$, the final 
$\tri (\bsigma-\bsigma_h,\bu - \bu_h )\tri_{\eg,\theta}$,  
and the final interpolation error 
$\tri (\bsigma- I^{rt}_{h,\nu}\bsigma,
\bu - I^{l2}_h\bu )\tri_{\full,\theta}$
with $\bRT_{0,\nu}\times S^d_{1,0}$ and $\theta_K = h^2_K$.}
\label{tab_AugRTh_compare}
\begin{tabular}{|c|c|c|c|c|c|c|}
\hline
&$\mbox{ind-err}$& $k$  & $n$  & 
$ \tri (\bsigma-\bsigma_h,\bu - \bu_h )\tri_{\eg,\theta}$ &
$\tri (\bsigma- I^{rt}_{h,\nu}\bsigma,
\bu - I^{l2}_h\bu )\tri_{\full,\theta}$\\\hline
$\mbox{Data}1$ & $0.9798$& $57$  &  $3896$ & $0.4665$ & $0.4761$  \\\hline
$\mbox{Data}2$ & $0.9550$& $43$  &  $4114$ & $0.2237$ & $0.2343$  \\\hline
$\mbox{Data}3$ & $0.9454$& $29$  &  $3002$ & $0.1675$ & $0.1772$ \\\hline
$\mbox{Data}4$ & $0.9308$& $21$  &  $2084$ & $0.2160$ & $0.2320$ \\\hline
$\mbox{Data}5$ & $0.9239$& $15$  &  $1256$ & $0.7333$ & $0.7937$  \\\hline
\end{tabular}
\end{table}

\begin{table}[htbp]
\caption{Index of the a prior estimate, number of refinements $k$, number of elements $n$, the final 
$\tri (\bsigma-\bsigma_h,\bu - \bu_h )\tri_{\eg,\theta}$,  
and the final interpolation error 
$\tri (\bsigma- I^{bdm}_{h,\nu}\bsigma,
\bu - I^{l2}_h\bu )\tri_{\full,\theta}$
with $\bBDM_{1,\nu}\times S^d_{2,0}$ and $\theta_K = h_K^2$.}
\label{tab_Augbdm_compare}
\begin{tabular}{|c|c|c|c|c|c|c|}
\hline
&$\mbox{ind-err}$& $k$  & $n$  & 
$ \tri (\bsigma-\bsigma_h,\bu - \bu_h )\tri_{\eg,\theta}$ &
$\tri (\bsigma- I^{bdm}_{h,\nu}\bsigma,
\bu - I^{l2}_h\bu )\tri_{\full,\theta}$ \\\hline
$\mbox{Data}1$ & $0.9820$& $101$  &  $4140$ & $0.0605$ & $0.0616$ \\\hline
$\mbox{Data}2$ & $0.9761$& $63$  &  $2528$ & $0.0382$ & $0.0391$  \\\hline
$\mbox{Data}3$ & $0.9762$& $40$  &  $1564$ & $0.0311$ & $0.0319$ \\\hline
$\mbox{Data}4$ & $0.9693$& $28$  &  $1048$ & $0.0410$ & $0.0423$\\\hline
$\mbox{Data}5$ & $0.9458$& $21$  &  $768$ & $0.1225$ & $0.1295$  \\\hline
\end{tabular}
\end{table}

\subsection{Adaptive convergence tests}
For the proposed methods, we present some adaptive convergence results.  We use the least-squares functional a posteriori error estimator \eqref{ap_lsfem}. 
The relative error is computed by 
\beq\notag
\mbox{rel-err} := \frac{\tri (\bsigma-\bsigma_h,\bu - \bu_h )\tri_{\eg,\theta}}
{\tri (\bsigma,\bu )\tri_{\eg,\theta}}.
\eeq

In the first part of the numerical tests, we test the methods with Data 1 to check the adaptive meshes the estimator generated and convergence histories.  
For the discrete problem \eqref{eq_varfrom_dis} with $\theta = 1$ and $\bRT_{0,\nu}\times S^2_{1,0}$,  the stopping criteria is given as $\mbox{rel-err}<0.11$ and the parameter in D\"{o}rfler marking strategy is $0.15$.  The decay of $Dof^{-1/2}$, $\eta_\theta$, $\tri(\bsigma-\bsigma_h, \bu - \bu_h)\tri_{\eg,\theta}$ are shown on the left of Figure \ref{fig_line_mesh_Data1} and the right of Figure \ref{fig_line_mesh_Data1} shows the meshes with $2228$ degrees of freedom at the final ($49$-th) loop. \begin{figure}[htbp]
  \centering
  \begin{minipage}[t]{0.45\linewidth}
  \includegraphics[scale=0.5]{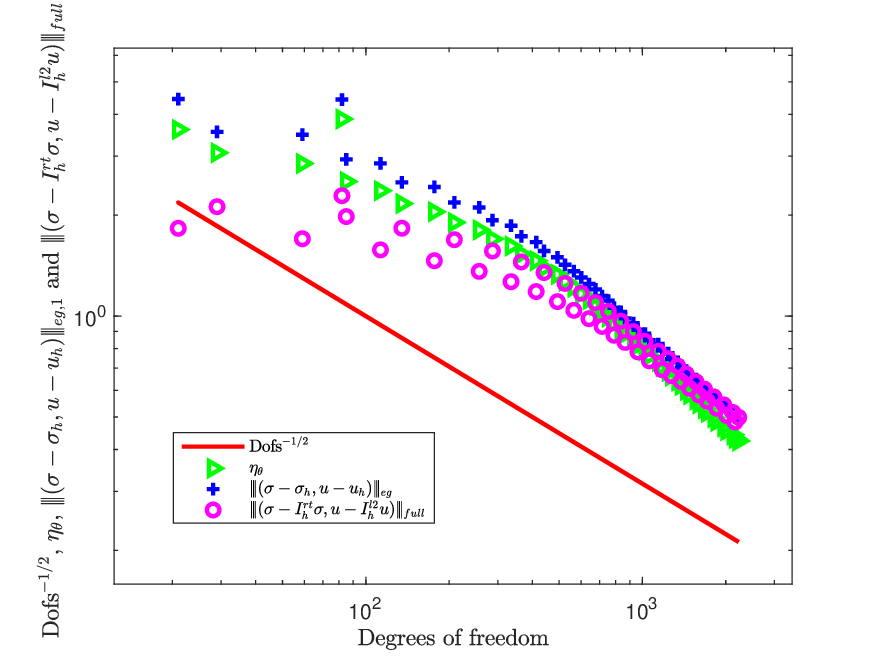}   
  \end{minipage}
    \begin{minipage}[t]{0.45\linewidth}
  \includegraphics[scale=0.5]{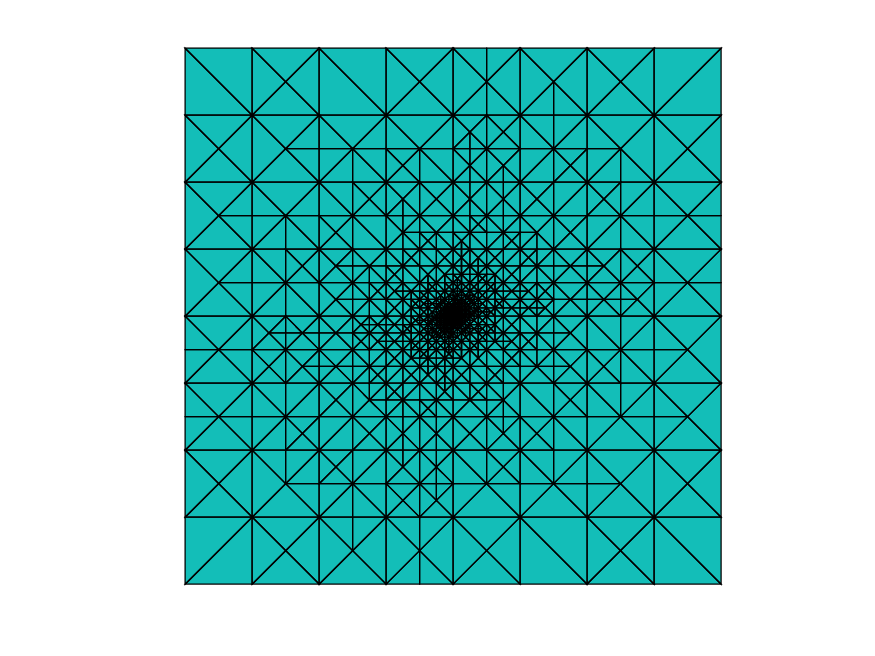}   
  \end{minipage}
  \caption{Adaptive results with Data 1, $\theta|_K = 1$, $RT_{0,\nu}\times S^d_{1,0}$: The reference line, $\eta_\theta,$ energy error, error by interpolation, and the final mesh}
    \label{fig_line_mesh_Data1}
\end{figure}
For the discrete problem \eqref{eq_varfrom_dis}  with $\theta_K = h_K^2$  and $ \bRT_{0,\nu}\times S^2_{1,0}$,   the stopping criteria is given as $\mbox{rel-err}<0.11$ and the parameter in D\"{o}rfler marking strategy is $0.15$.  The decay of $Dof^{-1/2},\eta_\theta, \tri(\bsigma-\bsigma_h, \bu - \bu_h)\tri_{\eg,\theta}$ are shown on the left of Figure \ref{fig_line_mesh_Data1_h} and the right of Figure \ref{fig_line_mesh_Data1_h} shows the meshes with $1969$ degrees of freedom at the final  ($57$-th) loop.
\begin{figure}[htbp]
  \centering
  \begin{minipage}[t]{0.45\linewidth}
  \includegraphics[scale=0.5]{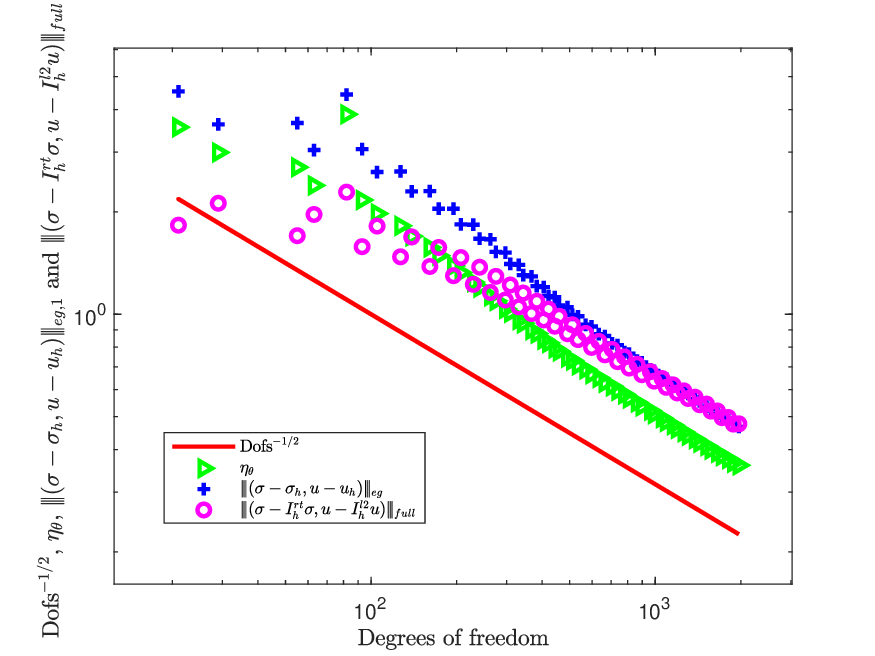}   
  \end{minipage}
    \begin{minipage}[t]{0.45\linewidth}
  \includegraphics[scale=0.5]{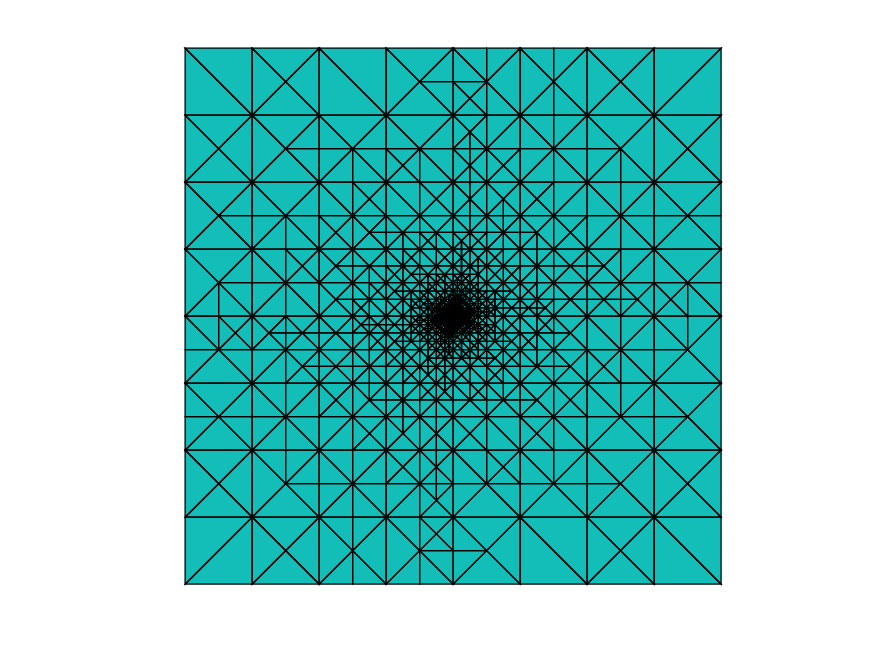}   
  \end{minipage}
  \caption{Adaptive results with Data 1, $\theta|_K = h_K$, $RT_{0,\nu}\times S^d_{1,0}$: The reference line, $\eta_\theta,$ energy error, error by interpolation, and the final mesh}
    \label{fig_line_mesh_Data1_h}
\end{figure}
For the case of $\bBDM_{1,\nu}\times S^d_{2,0}$ and $\theta_K = h_K^2$,  the stopping criteria is $\mbox{rel-err}<0.08$ and the parameter in D\"{o}fler marking strategy is $0.15$. The decay of $Dof^{-1/2},\eta_\theta, \tri(\bsigma-\bsigma_h, \bu - \bu_h)\tri_{\eg,\theta}$ are shown on the left of  Figure \ref{fig_line_mesh_Data1_h_BDM} and the right of Figure \ref{fig_line_mesh_Data1_h_BDM} shows the meshes with $465$ degrees of freedom the final  ($49$-th) loop.

All three numerical tests show optimal convergence and the final meshes are similar to the classical Kellogg's test in \cite{CD:02,CZ:09,CZ:10a,CZ:12,CHZ:17,CHZ:17zz,CHZ:21} and the decay of the errors are optimal.

\begin{figure}[htbp]
  \centering
  \begin{minipage}[t]{0.45\linewidth}
  \includegraphics[scale=0.5]{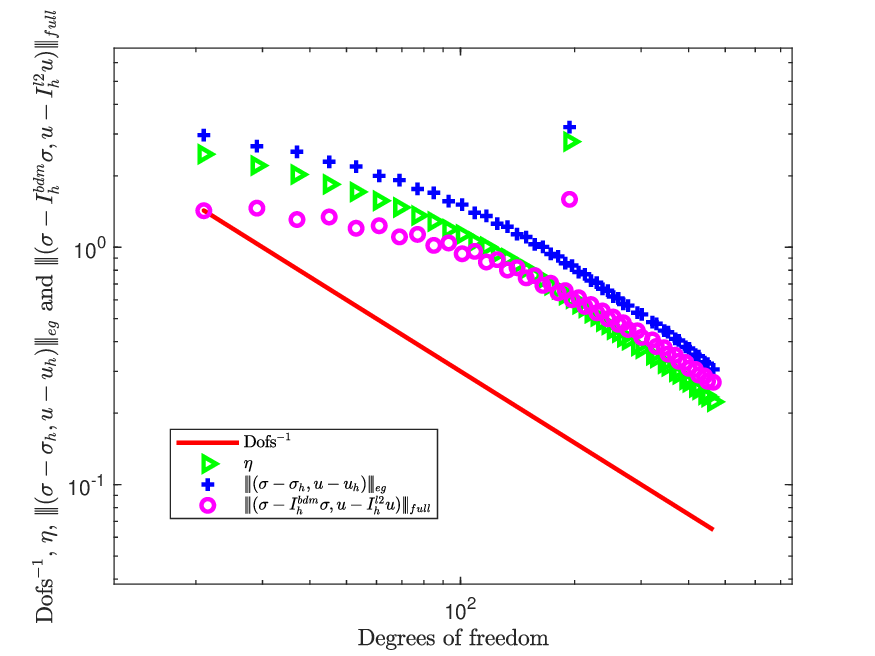}   
  \end{minipage}
    \begin{minipage}[t]{0.45\linewidth}
  \includegraphics[scale=0.5]{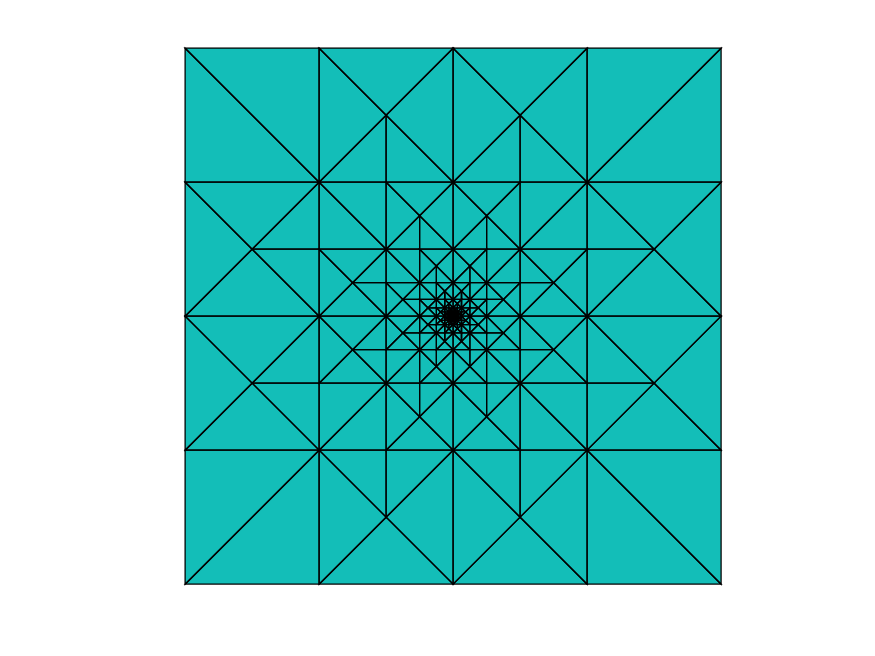}   
  \end{minipage}
  \caption{The reference line, $\eta_\theta,$ energy error, error by interpolation and the meshes of the example Data 1, with $\theta|_K = h^2_K$, using  $BDM_{1,\nu}\times S^d_{2,0}.$}
    \label{fig_line_mesh_Data1_h_BDM}
\end{figure}

\subsection{Tests on the robustness of a posteriori error estimators}
For the proposed methods, we test the robustness of the least-squares functional a posteriori error estimator \eqref{ap_lsfem} proposed in Section 7 even though there is no robustness proof.
The effectivity index is computed as 
\beq\notag
\mbox{eff-index} :=  
\frac{\tri(\bsigma - \bsigma_h 
, \bu - \bu_h)\tri_{\eg,\theta}}
{\eta_\theta}.
\eeq
If the error estimator is robust, then the $\mbox{eff-index}$ should almost be a constant for different choices of $\nu$.

The following Table \ref{tab_adap_theta1RT}-Table \ref{tab_adap_theta2Bdm} show the effectivity indexes of three different formulations with the exact solutions of Data1-Data5. The parameter in the D\"ofler marking strategy is $0.15$. The stopping criteria $\mbox{rel-err}<0.11$ is used for Table \ref{tab_adap_theta1RT} and Table \ref{tab_adap_theta2RT}  and the stopping criteria $\mbox{rel-err} < 0.02$ is used for Table \ref{tab_adap_theta2Bdm}. As shown in these tables, the $\mbox{eff-index}$ are  almost a constant for different choices of $\nu$ for each case. Thus, the least-squares functional a posteriori error estimators are computationally robust even though there is no rigorously proof yet.

\begin{table}[htbp]
\caption{The discrete problem \eqref{eq_varfrom_dis}: eff-index, number of refinements $k$, number of elements $n$, the final $\eta_\theta$, and the final error $\tri(\bsigma-\bsigma_h
,\bu-\bu_h)\tri_{eg,\theta},$
with $\bRT_{0,\nu}\times S^d_{1,0}$ and $\theta_K = 1.$}
\label{tab_adap_theta1RT}
\begin{tabular}{|c|c|c|c|c|c|}
\hline
&$\mbox{eff-index}$& $k$  & $n$  & $\eta_\theta$ &
$\tri(\bsigma-\bsigma_h
,u-u_h)\tri_{\eg,\theta}$\\\hline
$\mbox{Data}1$ & $1.1737$& $49$  &  $4388$ & $0.4228$ & $0.4962$ \\\hline
$\mbox{Data}2$ & $1.2099$& $39$  &  $4548$ & $0.1883$ & $0.2279$ \\\hline
$\mbox{Data}3$ & $1.2123$& $27$  &  $3318$ & $0.1383$& $0.1677$ \\\hline
$\mbox{Data}4$ & $1.2082$& $19$  &  $2064$ & $0.1891$& $0.2285$ \\\hline
$\mbox{Data}5$ & $1.1909$& $14$  &  $1296$ & $0.6250$& $0.7443$ \\\hline
\end{tabular}
\end{table}

\begin{table}[htbp]
\caption{The discrete problem \eqref{eq_varfrom_dis}: eff-index, number of refinements $k$, number of elements $n$, the final $\eta_\theta$, and the final error $\tri(\bsigma-\bsigma_h
,\bu-\bu_h)\tri_{eg,\theta},$
with $\bRT_{0,\nu}\times S^d_{1,0}$ and $\theta_K = h^2_K.$}
\label{tab_adap_theta2RT}
\begin{tabular}{|c|c|c|c|c|c|}
\hline
&$\mbox{eff-index}$& $k$  & $n$  & $\eta_\theta$ &
$\tri(\bsigma-\bsigma_h
,u-u_h)\tri_{\eg,\theta}$\\\hline
$\mbox{Data}1$ & $1.2973$& $57$  &  $3896$ & $0.3596$ & $0.4665$ \\\hline
$\mbox{Data}2$ & $1.2717$& $43$  &  $4114$ & $0.1759$ & $0.2237$ \\\hline
$\mbox{Data}3$ & $1.2512$& $29$  &  $3002$ & 
$0.1339$& $0.1675$ \\\hline
$\mbox{Data}4$ & $1.2413$& $21$  &  $2084$ & $0.1740$& $0.2160$ \\\hline
$\mbox{Data}5$ & $1.2346$& $15$  &  $1256$ & $0.5939$& $0.7333$ \\\hline
\end{tabular}
\end{table}

\begin{table}[htbp]
\caption{The discrete problem \eqref{eq_varfrom_dis}: eff-index, number of refinements $k$, number of elements $n$, the final $\eta_\theta$, and the final error $\tri(\bsigma-\bsigma_h
,\bu-\bu_h)\tri_{eg,\theta},$
with $\bBDM_{1,\nu}\times S^d_{2,0}$ and $\theta_K = h^2_K.$}
\label{tab_adap_theta2Bdm}
\begin{tabular}{|c|c|c|c|c|c|}
\hline
&$\mbox{eff-index}$& $k$  & $n$  & $\eta_\theta$ &
$\tri(\bsigma-\bsigma_h
,u-u_h)\tri_{\eg,\theta}$\\\hline
$\mbox{Data}1$ & $1.3586$& $101$  &  $4140$ & $0.0445$ & $0.0605$ \\\hline
$\mbox{Data}2$ & $1.3293$& $63$  &  $2528$ & $0.0287$ & $0.0382$ \\\hline
$\mbox{Data}3$ & $1.3038$& $40$  &  $1564$ & 
$0.0239$& $0.0311$ \\\hline
$\mbox{Data}4$ & $1.2751$& $28$  &  $1048$ & $0.0321$& $0.0410$ \\\hline
$\mbox{Data}5$ & $1.2417$& $21$  &  $768$ & $0.0986$& $0.1225$ \\\hline
\end{tabular}
\end{table}

\section{Concluding Remarks}
In this paper, we propose a robust augmented mixed finite element formulation for the stationary Stokes problem with a piecewise constant viscosity coefficient in multiple subdomains. This is the first one of such methods for the case of multiple subdomains. The main idea is that with augmented mixed methods, the inf-sup condition can be avoided. In the robust continuity, we use two norms, one energy norm and one full norm, to get the upper bound. Additionally, we introduce a Kellogg-type example with exact solutions for the first time, which will be useful as a benchmark test problem for Stokes problems with singularities.  

As mentioned in the paper, there exist alternative pathways to achieve robustness without relying on  inf-sup conditions. One approach involves the utilization of (strongly or weakly) divergence-free finite elements, while another way  explores the application of different stabilization techniques.  We will pursue these directions in our future research.

\bibliographystyle{plain}
\bibliography{../bib/szhang}
\section{Appendix: Nonlinear system to generate the Kellogg-type example}
\setcounter{equation}{0}
In this appendix, we present a detailed derivation of the nonlinear system consisting 16 equations used in Section \ref{kellogg} to construct the Kellogg-type example. By \eqref{BB}, we have
\beq 
B_1|_{\O_i} = r^\a(\ma_i\sin(\a\vartheta) + \mb_i\cos(\a\vartheta) ) 
\quad\mbox{and}\quad B_2|_{\O_i} = r^\a(\mc_i\sin(\a\vartheta) + \md_i\cos(\a\vartheta) )
\quad i=1,\cdots,4.
\eeq
To compute $\bu = \grad(\bx\cdot{\bf B}) - 2{\bf B}$, we first compute derivatives, 
\begin{eqnarray*}
\p_{x}B_1|_{\O_i}
= \a r^{\a-1}( \ma_i\sin(\a-1)\vartheta + \mb_i\cos(\a-1)\vartheta ), \quad
\p_{y}B_1|_{\O_i}
= \a r^{\a-1}( \ma_i\cos(\a-1)\vartheta - \mb_i\sin(\a-1)\vartheta ),\\
\p_{x}B_2|_{\O_i}
= \a r^{\a-1}( \mc_i\sin(\a-1)\vartheta + \md_i\cos(\a-1)\vartheta ),  \quad
\p_{y}B_2|_{\O_i}
= \a r^{\a-1}( \mc_i\cos(\a-1)\vartheta - \md_i\sin(\a-1)\vartheta ).
\end{eqnarray*}
Then we have, 
\begin{eqnarray*}
\p_{x}(\bx\cdot{\bf B})|_{\O_i}
&=& \a r^{\a-1}[ (\ma_ix + \mc_iy)\sin(\a-1)\vartheta
+ (\mb_ix + \md_iy)\cos(\a-1)\vartheta ] + B_1|_{\O_i},\\
\p_{y}(\bx\cdot{\bf B})|_{\O_i}
&=& \a r^{\a-1}[ (\ma_ix + \mc_iy)\cos(\a-1)\vartheta
- (\mb_ix + \md_iy)\sin(\a-1)\vartheta ] + B_2|_{\O_i}.
\end{eqnarray*}
Thus by $\bu = \grad(\bx\cdot{\bf B}) - 2{\bf B} = (u_1,u_2)^t$ in each $\O_i$, we have
\begin{eqnarray}\notag
u_1|_{\O_i} &=& 
\p_{x}(\bx\cdot{\bf B})|_{\O_i}
- 2 B_1|_{\O_i} \\ \notag
&=& \a r^{\a-1}[ (\ma_ix + \mc_iy)\sin(\a-1)\vartheta
+ (\mb_ix + \md_i y)\cos(\a-1)\vartheta ] - B_1|_{\O_i},\\\label{eq_u1xy}
&=& \a r^\a[ (\ma_i\cos\vartheta + \mc_i\sin\vartheta)\sin(\a-1)\vartheta
+ (\mb_i\cos\vartheta + \md_i\sin\vartheta)\cos(\a-1)\vartheta ] - B_1|_{\O_i}\\ \notag
u_2|_{\O_i} &=& 
\p_{y}(\bx\cdot{\bf B})|_{\O_i}- 2 B_2|_{\O_i} \\ \notag
&=& \a r^{\a-1}[ (\ma_ix + \mc_iy)\cos(\a-1)\vartheta
- (\mb_ix + \md_iy)\sin(\a-1)\vartheta ] - B_2|_{\O_i},\\ \label{eq_u2xy}
&=& \a r^\a[ (\ma_i\cos\vartheta + \mc_i\sin\vartheta)\cos(\a-1)\vartheta
- (\mb_i\cos\vartheta + \md_i\sin\vartheta)\sin(\a-1)\vartheta ] - B_2|_{\O_i}.
\end{eqnarray}
To compute $\bsigma$, we first compute the following derivatives from \eqref{eq_u1xy} and \eqref{eq_u2xy}:
\begin{eqnarray*}
\p_{x} u_1|_{\O_i}
&=& \a(\a-1) r^{\a-1}[ (\ma_i\cos\vartheta + \mc_i\sin\vartheta)\sin(\a-2)\vartheta
+ (\mb_i\cos\vartheta + \md_i\sin\vartheta)\cos(\a-2)\vartheta ],\\
\p_{y} u_1|_{\O_i}
&=& \a(\a-1) r^{\a-1}[ (\ma_i\cos\vartheta + \mc_i\sin\vartheta)\cos(\a-2)\vartheta
- (\mb_i\cos\vartheta + \md_i\sin\vartheta)\sin(\a-2)\vartheta ]\\
&& + \p_{x} B_2|_{\O_i} - \p_{y} B_1|_{\O_i},\\
\p_{x} u_2|_{\O_i}
&=& \a(\a-1) r^{\a-1}[ (\ma_i\cos\vartheta + \mc_i\sin\vartheta)\cos(\a-2)\vartheta
- (\mb_i\cos\vartheta + \md_i\sin\vartheta)\sin(\a-2)\vartheta ]\\
&& + \p_{y} B_1|_{\O_i} - \p_{x} B_2|_{\O_i},\\
\p_{y} u_2|_{\O_i}
&=& -\a(\a-1) r^{\a-1}[ (\ma_i\cos\vartheta + \mc_i\sin\vartheta)\sin(\a-2)\vartheta
+ (\mb_i\cos\vartheta + \md_i\sin\vartheta)\cos(\a-2)\vartheta ].
\end{eqnarray*}
The weighted pressure $\nu_i^{-1} p$ is, 
\begin{eqnarray*}
\nu_i^{-1} p|_{\O_i} = \gradt{\bf B}|_{\O_i}
= \a r^{\a-1}( \ma_i\sin(\a-1)\vartheta + \mb_i\cos(\a-1)\vartheta
+ \mc_i\cos(\a-1)\vartheta - \md_i\sin(\a-1)\vartheta)
\end{eqnarray*}
By $\bsigma = \nu\beps(\bu)-pI$, for $i=1,\cdots,4$, we have 
\begin{eqnarray*}
\sigma_{11}|_{\O_i} =  \nu_i \p_{x} u_1|_{\O_i}
- \nu_i\gradt{\bf B}|_{\O_i},
\sigma_{12}|_{\O_i} = \sigma_{21}|_{\O_i}
= \nu_i(\p_{y} u_1|_{\O_i} + \p_{x} u_2|_{\O_i})/2,
\sigma_{22}|_{\O_i} = 
\nu_i \p_{y} u_2|_{\O_i} -\nu_i\gradt{\bf B}|_{\O_i}.
\end{eqnarray*}

%
%

The continuous condition $\jump{u_1} = 0$ and $\jump{u_2} = 0$ on the interface for $\vartheta = \pi/2,\pi,\frac{3\pi}2$ and $2\pi$ gives the following 8 equations,
\begin{eqnarray}
 \label{eq01}
-[\a(\mc_1-\mc_2) + (\mb_1-\mb_2)]\cos\frac{\pi}2\a 
+ [ \a(\md_1-\md_2) - (\ma_1-\ma_2)]\sin\frac{\pi}2\a  &=& 0,\\
(\ma_2-\ma_3)\sin\pi\a + (\mb_2-\mb_3)\cos\pi\a &=& 0,\\
-[\a(\mc_3-\mc_4) + (\mb_3-\mb_4)]\cos\frac{3\pi}2\a  
+ [ \a(\md_3-\md_4) - (\ma_3-\ma_4)]\sin\frac{3\pi}2\a  &=& 0,\\
\ma_4\sin2\pi\a + \mb_4\cos2\pi\a - \mb_1 &=& 0,\\
(\mc_1 - \mc_2)\sin\frac{\pi}2\a + (\md_1 - \md_2)\cos\frac{\pi}2\a  &=& 0,\\
{[} \a (\ma_2 - \ma_3) - (\md_2 - \md_3) ]\cos\pi\a 
- [ \a (\mb_2 - \mb_3) + (\mc_2 - \mc_3) ]\sin\pi\a &=& 0,\\
(\mc_3 - \mc_4)\sin\frac{3\pi}2\a + (\md_3 - \md_4)\cos\frac{3\pi}2\a &=& 0,\\
(\a \ma_4 - \md_4)\cos2\pi\a - (\a \mb_4 + \mc_4)\sin2\pi\a - (\a \ma_1 - \md_1) &=& 0.
\end{eqnarray}
The normal jump of $\bsigma$ over interfaces reads
$\jump{\bsigma_{11}}|_{\vartheta = \pi/2}
= \jump{\bsigma_{12}}|_{\vartheta = \pi}
= \jump{\bsigma_{11}}|_{\vartheta = 3\pi/2}
= \jump{\bsigma_{12}}|_{\vartheta = 2\pi} = 0$
and $\jump{\bsigma_{21}}|_{\vartheta = \pi/2}
= \jump{\bsigma_{22}}|_{\vartheta = \pi}
= \jump{\bsigma_{21}}|_{\vartheta = 3\pi/2}
= \jump{\bsigma_{22}}|_{\vartheta = 2\pi} = 0$.
We have the following 8 equations:
\begin{eqnarray}
\nu_1[ (\a \mc_1+\mb_1)\sin\frac{\pi}2\a + (\a \md_1-\ma_1)\cos\frac{\pi}2\a ]
&=& \nu_2[ (\a \mc_2+\mb_2)\sin\frac{\pi}2\a + (\a \md_2-\ma_2)\cos\frac{\pi}2\a ],\\
\nu_2( \ma_2\cos\pi\a - \mb_2\sin\pi\a )
&=& \nu_3( \ma_3\cos\pi\a - \mb_3\sin\pi\a ),\\
\nu_3[ (\a \mc_3+\mb_3)\sin\frac{3\pi}2\a + (\a \md_3-\ma_3)\cos\frac{3\pi}2\a ]
&=& \nu_4[ (\a \mc_4+\mb_4)\sin\frac{3\pi}2\a + (\a \md_4-\ma_4)\cos\frac{3\pi}2\a ],\\
\nu_4( \ma_4\cos2\pi\a - \mb_4\sin2\pi\a )
&=& \nu_1 \ma_1,\\
\nu_1( \mc_1\cos\frac{\pi}2\a - \md_1\sin\frac{\pi}2\a )
&=& \nu_2( \mc_2\cos\frac{\pi}2\a - \md_2\sin\frac{\pi}2\a ),\\
\nu_2[ (\a \ma_2 - \md_2)\sin\pi\a + (\a \mb_2 + \mc_2)\cos\pi\a ]
&=& \nu_3[ (\a \ma_3 - \md_3)\sin\pi\a + (\a \mb_3 + \mc_3)\cos\pi\a ],\\
\nu_3( \mc_3\cos\frac{3\pi}2\a - \md_3\sin\frac{3\pi}2\a )
&=& \nu_4( \mc_4\cos\frac{3\pi}2\a - \md_4\sin\frac{3\pi}2\a ),\\ \label{eq16}
\nu_4[ (\a \ma_4 - \md_4)\sin2\pi\a + (\a \mb_4 + \mc_4)\cos2\pi\a ]
&=& \nu_1(\a \mb_1 + \mc_1).
\end{eqnarray}
Thus, we get the nonlinear system with 16 equations \eqref{eq01}-\eqref{eq16} for $\a$, $\nu_i$,  $\ma_i,\mb_i,\mc_i$ and $\md_i$ for $i=1,\cdots,4$.
\end{document}